\documentclass[12pt]{extarticle} % The extarticle class with 12pt font size

\usepackage{bbm}
\usepackage{tikz} 

\usepackage[a4paper, total={7in, 9in}]{geometry} 

\usepackage{pgfplots}
\usepackage{graphics}
\usepackage{float}

\usepackage[margin=1cm]{caption}
\usepackage{color}

\usepackage{enumerate} 
\usepackage{algorithm, algorithmic}

\usepackage[noadjust]{cite}
\usepackage{authblk}

\usepackage{setspace}
\usepackage{amsfonts, amsthm, amsmath, latexsym, amssymb}
\newtheorem{theorem}{Theorem}[section]
\newtheorem{lemma}[theorem]{Lemma}

\newtheorem{proposition}[theorem]{Proposition}
\newtheorem{corollary}[theorem]{Corollary}

\theoremstyle{definition}
\newtheorem{definition}[theorem]{Definition}
\theoremstyle{remark}
\newtheorem{remark}[theorem]{Remark}
\numberwithin{equation}{section}

\newcommand{\cE}{\mathcal{E}}

\newcommand{\cN}{\mathcal{N}}
\newcommand{\cF}{\mathcal{F}}
\newcommand{\cB}{\mathcal{B}}

\newcommand{\cP}{\mathcal{P}}
\newcommand{\cC}{\mathcal{C}}
\newcommand{\cD}{\mathcal{D}}
\newcommand{\cI}{\mathcal{I}}

\newcommand{\bP}{\mathbb{P}}

\newcommand{\bE}{\mathbb{E}}

\newcommand{\R}{\mathbb{R}}

\newcommand{\la}{\lambda }

\newcommand{\si}{\sigma }

\newcommand{\ga}{\gamma }
\newcommand{\Ga}{\Gamma }

\newcommand{\ones}{\mathbbm{1}}
\newcommand{\one}{\mathbf{1}}
\newcommand{\Mod}{\operatorname{Mod}}
\newcommand{\Dom}{\operatorname{Dom}}

\newcommand{\Adm}{\operatorname{Adm}}
\newcommand{\Ext}{\operatorname{Ext}}
\newcommand{\MEO}{\operatorname{MEO}}

\newcommand{\Var}{\operatorname{Var}}
\newcommand{\co}{\operatorname{co}}

\newcommand{\cl}{\operatorname{cl}}
\newcommand{\bi}{\begin{itemize}}
	\newcommand{\ei}{\end{itemize}}
\newcommand{\Gahat}{\widehat{\Ga}}
\newcommand{\Gatil}{\widetilde{\Ga}}
\newcommand{\BL}{\operatorname{BL}}
\newcommand{\lbr}{\left\{ }
\newcommand{\rbr}{\right\} }
\newcommand{\argmax}{\operatorname{argmax}}
\pgfplotsset{compat=1.18}

\usepackage{comment}

\title{\bf Modulus for bases of matroids\thanks{This material is based upon work supported by the National Science Foundation under Grant n. 2154032.}}
\author[1]{Huy Truong}
\author[1]{Pietro Poggi-Corradini}
\affil[1]{\small Dept. of Mathematics, Kansas State University, Manhattan, KS 66506, USA.}
\date{}

\begin{document}
\maketitle
	
\begin{abstract}
		
In this work, we explore the application of modulus in matroid theory, specifically, the modulus of the family of bases of matroids. This study not only recovers  various concepts in matroid theory, including the strength, fractional arboricity, and principal partitions, but also offers new insights. In the process, we introduce the concept of a Beurling set. Additionally, our study revisits and provides an alternative approach to two of Edmonds's theorems related to the base packing and base covering problems. This is our stepping stone for establishing Fulkerson modulus duality for the family of bases. Finally, we provide a relationship between the base modulus of matroids and their dual matroids, and a complete understanding of the base $p$-modulus across all values of $p$.
\end{abstract}

\vspace{0.1in}

%%%%%%%%%%%%%%%%%%%%%%%%%%%%%%%%%%%%%%%%%%%%%%%%%%%%%%%%
\section{Introduction}

The theory of modulus on graphs has been extensively studied in recent years \cite{modulus, pietrominimal, pietroblocking, pietrofairest, polynomial}. Discrete modulus  is a very flexible and general tool for measuring the richness of families of objects defined on a finite set.
Consider a graph $G =(V,E)$, a family of objects is  defined as a collection of usage vectors on the edgeset $E$. In particular, subsets of $E$ are associated to their indicator functions in $\R^E$. Various families have been thoroughly investigated, such as all $s$-$t$ paths, all $s$-$t$ cuts, and all spanning trees. The modulus of the family of all $s$-$t$ paths recovers minimum $s$-$t$ cuts, shortest $s$-$t$ paths, and the effective resistance from $s$ to $t$ \cite{modulus}.
	
The modulus of the family of spanning trees on undirected graphs was studied in \cite{pietrofairest}. Their work involves the probabilistic interpretation of the spanning tree modulus: 
\begin{equation}\label{eq:co}
		\min \left\{ \sum\limits_{e \in E} \eta(e)^2  :\eta \in \co(\Ga) \right\},
\end{equation}
where $\Ga$ is the spanning tree family of an undirected graph $G=(V,E)$, and $\co(\Ga)$ is the convex hull of indicator vectors of all spanning trees in $\Ga$. The authors in \cite{pietrofairest} show that the optimal density $\eta^*$ of the problem (\ref{eq:co}) is closely related to two well-investigated concepts in  graph theory, namely the strength \cite{gusfieldtree,cunninghamoptimal,cunninghamfaster, catlin1992} and the fractional arboricity of graphs \cite{fastapproximation,hong2016fractional, catlin1992}, and they also propose a deflation process for graphs that identifies a hierarchical structure of graphs.\\

The concepts of strength and fractional arboricity of graphs are actually special cases of the corresponding concepts for matroids. Let us revisit these notions. For a loopless matroid $M(E,\cI)$ on the ground set $E$ with a family of independent sets $\cI$ and rank function $r$, the {\it density} $\theta(M)$ of $M$ is defined as $|E|/r(E)$, the {\it strength} of $M$ is defined as:
\begin{equation}\label{eq:strength-problem}
S(M) := \min \left\{ \frac{|X|}{r(E) - r(E-X)} : X \subseteq E, r(E) > r(E-X) \right\},
\end{equation}
and the {\it fractional arboricity} of $M$ is defined as:
\begin{equation}
D(M) := \max \left\{ \frac{|X|}{r(X)} : X \subseteq E, r(X)>0 \right\}.
\end{equation}

In \cite{catlin1992}, the authors present several characterizations of matroids (or graphs) $M$ for which $S(M) = D(M)$ \cite[Theorem 6]{catlin1992}. One such characterization is that there exists a positive integer $t$ and a family $\cF$ of bases (or spanning trees) of $M$, such that each element $e \in E$ appears in exactly $t$ bases in $\cF$. For the graph case, these graphs have been called by several names, including homogeneous graphs \cite{pietrofairest}, strongly balanced graphs \cite{rucinski1986strongly}, and uniformly dense graphs \cite{catlin1992}. In this paper, a matroid $M$ for which 
\begin{equation}\label{eq:homogeneous-matroid}
S(M) = D(M),
\end{equation}
is called a {\it homogeneous matroid}.\\
	
The theory of principal partitions in graphs, matroids, and submodular systems has been developed since 1968. The principal partition can be considered as a decomposition of a discrete system into its components together with a
partially ordered structure of the set of components.
For an overview of the theory of principal partitions, we recommend the survey paper \cite{fujishige2009theory}. In particular, the hierarchical structure of arbitrary graphs identified through the deflation process described in \cite{pietrofairest} coincides with the principal partitions of graphic matroids in \cite{fujishige2009theory}. The optimal density $\eta^*$ of the problem (\ref{eq:co}) is equal to both the universal base and the lexicographically optimal base of the graphic polymatroid associated with the graph $G$, see Section \ref{sec:pp} for definitions and \cite{fujishige1980lexicographically}.
	
Consider a loopless matroid $M(E,\cI)$ with the rank function $r$. In \cite{catlin1992}, the authors established that $S(M)$ and $D(M)$ are closely related to the principal partition of $M$. The theory of principal partitions is not only a powerful tool but also has many applications, as discussed in \cite{fujishige2009theory}. Specifically, this theory shows that $S(M) = 1/ \eta^*_{max}$ and $D(M) = 1/ \eta^*_{min}$, where $\eta^*_{max}$ and $\eta^*_{min}$ are the maximum and minimum values, respectively, of the universal base of the polymatroid associated with $M$, see Section \ref{sec:pp}.

It is noteworthy that there is a fundamental relationship between strength and arboricity of a matroid and its dual, see \cite[Theorem 1]{catlin1992}. This relationship motivates one of our key results, detailed in Theorem \ref{thm:dualmat} below.\\
	
The theory of blocking and anti-blocking polyhedra was developed by Fulkerson in \cite{fulkersonblocking, fulkersonanti}. This theory describes dual relationships 
between sets of constraints in optimization problems, with applications in 
network flow and combinatorial optimization.
Let $\cB$ be the family of all bases of a loopless matroid $M = (E, \cI)$. Let $\co(\cB)$ be the convex hull of indicator vectors of all bases in $\cB$. The {\it dominant} of $\cB$ is defined as:
\begin{equation}\label{eq:dominant-bases}
\Dom(\cB) := \co(\cB) + \R^{E}_{\geq 0}.
\end{equation}
The base dominant $\Dom(\cB)$ is discussed in \cite{fulkersonanti}. Fulkerson stated in \cite{fulkersonanti} that the following result can be deduced from Edmonds's work in \cite{edmonds1965lehman}:
	
\begin{align}\label{eq:edmonsnash}
\Dom(\cB) = \left\{ \eta \in \R^E_{\geq 0} : \sum\limits_{e\in X} \eta(e) \geq r(E) - r(E-X), \text{ for all closed } X  \subseteq E \right\}.
\end{align}
	
In \cite{edmonds1965lehman}, Edmonds presents two theorems from which it can be derived that $S(M)$ is equal to the value of the base packing problem (which is equivalent to (\ref{eq:edmonsnash})), and $D(M)$ is equal to the value of the base covering problem,  see Section \ref{sec:pp} for definitions and \cite{catlin1992}. These theorems extend results from Nash-Williams and Tutte in \cite{edge-disjoint,on} and Nash-Williams in \cite{nash1964decomposition}. Specifically, for the case of the spanning tree dominant, Chopra in \cite{chopraon} provided the same inequality description as in (\ref{eq:edmonsnash}),  using a different proof. Moreover, he also provided a minimal inequality description of the spanning tree dominant.
	
In this paper, we investigate the modulus of the family of all bases of matroids, which we refer to as {\it  base modulus}. Here is a summary of our results.
\bi
\item	 In Section \ref{sec:basemodulus}, we analyze the base modulus and its probabilistic interpretation. In the process, we introduce the concept of Beurling sets.

\item In Section \ref{sec:strength}, we generalize several results from spanning tree modulus to base modulus. In particular, some of these results recover various properties of matroids related to their strength, fractional arboricity, and principal partition. 
	
\item In Section \ref{sec:basepacking}, we present an alternative proof strategy to two of Edmonds's theorems in \cite{edmonds1965lehman}, utilizing results from our study of base modulus. 
	
\item In Section \ref{sec:Fulkerson}, we provide the Fulkerson blocker family of the base family of a matroid. In other words, we derive a minimal inequality description of the base dominant. %$\Adm(\widehat{\cB}) = \Dom(\cB)$
	
\item In Section \ref{sec:dualmatroid}, we establish a relationship between the base modulus of matroids and their dual matroids, enriching our understanding of the connection between a matroid and its dual.
\item In Section \ref{sec:otherp}, we demonstrate that $p$-modulus for the base family of a matroid can be deduced from the optimal density  of the $2$-modulus, and this provides a complete understanding of the base modulus across all values of $p$.
\ei

%BEGIN_FOLD Picture
%\begin{figure}[H]
\begin{figure}[t]
	\begin{center}
			\begin{tikzpicture}[scale=1.3, auto, node distance=3cm,  thin]
				\begin{scope}[every node/.style={circle,draw=black,fill=black!100!,font=\sffamily\Large\bfseries}]
					
					\node (A) [scale=0.3]at (0,0) {};
					\node (B)[scale=0.3] at (1,0) {};
					\node (C)[scale=0.3] at (1.5,0.87) {};
					\node (D) [scale=0.3]at (1,1.73) {};
					\node (E)[scale=0.3] at (0,1.73) {};
					\node (F)[scale=0.3] at (-0.5,0.87) {};
					\node (G)[scale=0.3] at (0.25,0.43) {};
					\node (H) [scale=0.3]at (0.75,0.43) {};
					\node (I)[scale=0.3] at (1,0.87) {};
					\node (J)[scale=0.3] at (0.75,1.3) {};
					\node (K)[scale=0.3] at (0.25,1.3) {};
					\node (L)[scale=0.3] at (0,0.87) {};
					
					\node (A1) [scale=0.3]at (0+3,0) {};
					\node (B1)[scale=0.3] at (1+3,0) {};
					\node (C1)[scale=0.3] at (1.5+3,0.87) {};
					\node (D1) [scale=0.3]at (1+3,1.73) {};
					\node (E1)[scale=0.3] at (0+3,1.73) {};
					\node (F1)[scale=0.3] at (-0.5+3,0.87) {};
					\node (G1)[scale=0.3] at (0.25+3,0.43) {};
					\node (H1) [scale=0.3]at (0.75+3,0.43) {};
					\node (I1)[scale=0.3] at (1+3,0.87) {};
					\node (J1)[scale=0.3] at (0.75+3,1.3) {};
					\node (K1)[scale=0.3] at (0.25+3,1.3) {};
					\node (L1)[scale=0.3] at (0+3,0.87) {};
					
					\node (A2) [scale=0.3]at (0+1.5,0-2.6) {};
					\node (B2)[scale=0.3] at (1+1.5,0-2.6) {};
					\node (C2)[scale=0.3] at (1.5+1.5,0.87-2.6) {};
					\node (D2) [scale=0.3]at (1+1.5,1.73-2.6) {};
					\node (E2)[scale=0.3] at (0+1.5,1.73-2.6) {};
					\node (F2)[scale=0.3] at (-0.5+1.5,0.87-2.6) {};
					\node (G2)[scale=0.3] at (0.25+1.5,0.43-2.6) {};
					\node (H2) [scale=0.3]at (0.75+1.5,0.43-2.6) {};
					\node (I2)[scale=0.3] at (1+1.5,0.87-2.6) {};
					\node (J2)[scale=0.3] at (0.75+1.5,1.3-2.6) {};
					\node (K2)[scale=0.3] at (0.25+1.5,1.3-2.6) {};
					\node (L2)[scale=0.3] at (0+1.5,0.87-2.6) {};
					
				\end{scope}
				\begin{scope}[every edge/.style={draw=black,thin}]
					
					\draw  (G) edge node{} (H); 
					\draw  (H) edge node{} (I); 
					\draw  (I) edge node{} (J);
					\draw  (J) edge node{} (K);
					\draw  (K) edge node{} (L); 
					\draw  (L) edge node{} (G);

					\draw  (G1) edge node{} (H1); 
					\draw  (H1) edge node{} (I1); 
					\draw  (I1) edge node{} (J1);
					\draw  (J1) edge node{} (K1);
					\draw  (K1) edge node{} (L1); 
					\draw  (L1) edge node{} (G1);
					
					\draw  (G2) edge node{} (H2); 
					\draw  (H2) edge node{} (I2); 
					\draw  (I2) edge node{} (J2);
					\draw  (J2) edge node{} (K2);
					\draw  (K2) edge node{} (L2); 
					\draw  (L2) edge node{} (G2); 
					
					\draw  (G) edge node{} (I);
					\draw  (G) edge node{} (J);
					\draw  (G) edge node{} (K);
					\draw  (H) edge node{} (J);
					\draw  (H) edge node{} (K);
					\draw  (H) edge node{} (L);
					\draw  (I) edge node{} (K);
					\draw  (I) edge node{} (L);
					\draw  (J) edge node{} (L);

					\draw  (G1) edge node{} (I1);
					\draw  (G1) edge node{} (J1);
					\draw  (G1) edge node{} (K1);
					\draw  (H1) edge node{} (J1);
					\draw  (H1) edge node{} (K1);
					\draw  (H1) edge node{} (L1);
					\draw  (I1) edge node{} (K1);
					\draw  (I1) edge node{} (L1);
					\draw  (J1) edge node{} (L1);
					
					\draw  (G2) edge node{} (I2);
					\draw  (G2) edge node{} (J2);
					\draw  (G2) edge node{} (K2);
					\draw  (H2) edge node{} (J2);
					\draw  (H2) edge node{} (K2);
					\draw  (H2) edge node{} (L2);
					\draw  (I2) edge node{} (K2);
					\draw  (I2) edge node{} (L2);
					\draw  (J2) edge node{} (L2);
					
				\end{scope}
				
				\begin{scope}[every edge/.style={draw=black,dashed}]
					
					\draw  (A) edge node{} (B);
					\draw  (B) edge node{} (C);
					\draw  (C) edge node{} (D);
					\draw  (D) edge node{} (E);
					\draw  (E) edge node{} (F);
					\draw  (F) edge node{} (A);
					
					\draw  (A1) edge node{} (B1);
					\draw  (B1) edge node{} (C1);
					\draw  (C1) edge node{} (D1);
					\draw  (D1) edge node{} (E1);
					\draw  (E1) edge node{} (F1);
					\draw  (F1) edge node{} (A1);
					
					\draw  (A2) edge node{} (B2);
					\draw  (B2) edge node{} (C2);
					\draw  (C2) edge node{} (D2);
					\draw  (D2) edge node{} (E2);
					\draw  (E2) edge node{} (F2);
					\draw  (F2) edge node{} (A2);
					
					\draw  (A) edge node{} (G);
					\draw  (B) edge node{} (H);
					\draw  (C) edge node{} (I);
					\draw  (D) edge node{} (J);
					\draw  (E) edge node{} (K);
					\draw  (F) edge node{} (L);
					
					\draw  (A1) edge node{} (G1);
					\draw  (B1) edge node{} (H1);
					\draw  (C1) edge node{} (I1);
					\draw  (D1) edge node{} (J1);
					\draw  (E1) edge node{} (K1);
					\draw  (F1) edge node{} (L1);
					
					\draw  (A2) edge node{} (G2);
					\draw  (B2) edge node{} (H2);
					\draw  (C2) edge node{} (I2);
					\draw  (D2) edge node{} (J2);
					\draw  (E2) edge node{} (K2);
					\draw  (F2) edge node{} (L2);
					
				\end{scope}
				
				\begin{scope}[every edge/.style={draw=black,thick, dotted}]
					
					\draw  (C) edge node{} (F1);
					\draw  (B) edge node{} (E2);
					\draw  (D2) edge node{} (A1);
				\end{scope}
			\end{tikzpicture}
			
	\end{center}
	\caption{A graphic matroid where edges are styled according to the optimal density $\eta^*$. Solid edges are present with $\eta^* = \frac{1}{3}$, dashed edges with $\eta^* = \frac{1}{2}$ and dotted edges with  $\eta^* = \frac{2}{3}$.}\label{figure1}
\end{figure}

%END_FOLD

\textbf{Acknowledgment:} We would like to thank the anonymous referee for helpful comments and suggestions that improved the paper. We also would like to thank Nathan Albin for suggesting this study.

\section{Preliminaries}\label{sec:preliminaries}
\subsection{Matroids}
Let us begin by revisiting several definitions related to matroids. For a set $X$ we write $|X|$  for its cardinality, and if $Y$ is another set, then $X-Y$ is the relative complement of $Y$ in $X$.
\begin{definition}\label{def:independent-set}
Let $E$ be a finite set, the set system $M(E,\cI)$ is a matroid if the  following axioms are satisfied:
\bi
\item[(I1)] $\emptyset \in \cI$.
\item[(I2)] If $X \in \cI$ and $Y \subseteq X$ then $Y \in \cI$ ({\it Hereditary property}).
\item[(I3)] If $X,Y \in \cI$ and $|X| > |Y|$, then there exists $x \in X - Y$ such that $Y \cup \left\{x	\right\} \in \cI$ ({\it Exchange property}).
\ei
Every set in $\cI$ is called an {\it independent set}.
\end{definition}
Let  $M(E,\cI)$ be a matroid on the ground set $E$ with the set of  independent sets $\cI$. The maximal independent sets are called {\it bases},  the minimal dependent sets are called {\it circuits}. The {\it rank} function, $r : 2^E \rightarrow \mathbb{Z}_{+}$, defined on all subsets $X\subset E$ is given by:
\[r(X) := \max \left\{ |Y| : Y \subseteq X, Y \in \cI \right\}.\]
The {\it closure operator} $\cl: 2^E \rightarrow 2^E$ is a set function, defined as:
\begin{equation}\label{eq:closure-operator}
\cl(X) := \left\{ y \in E : r(X \cup \{ y \}) = r(X) \right\}.
\end{equation}
In matroid theory, these concepts play an important role. The following proposition gives some basic properties of matroids that we will use throughout the paper.
\begin{proposition}[\cite{mat}]\label{pro:basic}
Given a  matroid $M(E,\cI)$. Let $\cB$ be the set of bases of $M$ and let $\cC$ be the set of circuits of $M$. Let $r$ be the rank function.  Then, for all subsets $X,Y \subseteq E$, we have:
		\bi  
		\item[1.] If $B_1,B_2 \in\cB$ and $x\in B_1 - B_2$, then there exists $y \in B_2 - B_1$ such that $(B_1 - \lbr x \rbr ) \cup \lbr y \rbr \in \cB.$
		\item[2.] If $B \in \cB$ and $x \in E - B$, then there exists a unique circuit $C(x,B)$ contained in $B \cup \lbr x\rbr$ and containing $x$.
		\item[3.] If $B \in \cB,$ then for any $x \in E - B$, the set $(B - \lbr y \rbr ) \cup \lbr x \rbr$ is a base of $M$ if and only if $y \in C(x,B) $.
		\item[4.] $0 \leq r(X) \leq |X|$.
		\item[5.] If $Y \subseteq X$ then $r(Y) \leq r(X)$.
		\item[6.] $r(X) + r(Y) \geq r(X \cap Y) +r(X \cup Y)$.
		\item[7.] $X \in \cI \ \Leftrightarrow |X| = r(X).$
		\item[8.] $X \in \cB \ \Leftrightarrow |X| = r(X) = r(E).$
		\item[9.] $\cl(X) \supseteq X$.
		\item[10.] $\cl(X) = X \cup \lbr y \in E : y \in C \subseteq (X \cup y) \text{ for some } C \in \cC\rbr$ { \rm \cite{closure82}}.
		\ei
\end{proposition}
Next, we recall dual matroids. Given a matroid $M(E,\cI)$, the set
\[ \cB^* :=\{ X \subseteq E : \text{ there exists a base } B \in \cB \text{ such that } X = E - B \}\]
is the family of bases of the {\it dual matroid} $M^*$ on $E$. The {\it corank} function $r^*$ of $M$ is defined as the rank function of $M^*$, and for any subset $X \subseteq E$, we have:
\begin{equation}\label{eq:corank-rank}
r^*(X) = |X| - r(M) + r(E-X).
\end{equation}
Let us also recall the operations of {\it deletion}, {\it restriction}, and {\it contraction} in matroids. For a matroid $M(E,\cI)$ and a subset $X \subseteq E$, the set
\[ \cC(M \setminus X) = \{ C \subseteq E-X: C\in \cC(M) \},\]
defines the family of circuits for a matroid on $E-X$. The matroid $M \setminus X$ is called the {\it deletion} of $X$ from $M$. The {\it restriction} to $X$ in $M$ is denoted by $M|X$, and is the matroid on $X$ defined as $M|X := M \setminus (E - X)$.
The deletion operation behaves as in the following proposition.
\begin{proposition}[\cite{mat}]\label{pro:deletion}
Given a matroid $M(E,\cI)$ and a subset $X \subseteq E$, then:
\begin{enumerate}
\item $\cI(M \setminus X) = \{ Y \subseteq E -X : Y \in \cI(M)\}.$
\item $\cB(M \setminus X) = \text{maximal sets in } \{  B - X : B \in \cB(M) \}.$
\item $r_{M \setminus X}(Y) = r_{M}(Y)$ for all subsets $Y \subseteq E-X$.
\end{enumerate}
\end{proposition}	
Note that $r_M$ denotes the rank function of a matroid $M$. When we discuss a matroid $M$ and its deletion and contraction, $r$ is understood as the rank function of $M$.
	
For a matroid $M(E,\cI)$ and a subset $X \subseteq E$, the {\it contraction} of $X$ in $M$ is the matroid $M / X$ on $E - X$, which is defined as:
\[M/X = (M^* \setminus X)^*.\]
We have the following properties for the contraction operator:
	
\begin{proposition}[\cite{mat}]\label{pro:contraction}
Given a matroid $M(E,\cI)$ and a subset $X \subseteq E$, then:
\begin{enumerate}
			\item $\cC(M / X) = \text{minimal sets in }\lbr C - X : C \in \cC(M), C - X \neq \emptyset \rbr.$
			\item $\cI(M / X) = \lbr Y \subseteq E - X :\exists B \in \cB(M \setminus(E - X)) \text{ such that } Y \cup B \in \cI(M)\rbr.$
			\item $\cB(M / X) = \lbr Y \subseteq E - X : \exists B \in \cB(M \setminus(E - X))\text{ such that }Y \cup B \in \cB(M) \rbr.$
			\item $r_{M / X}(Y) = r_{M}(X \cup Y) - r_{M}(X)$ for all $Y \subseteq E - X$.
\end{enumerate}
\end{proposition}
Next, let us recall the definition of polymatroids. From now on, for any vector $x \in \R^E$ and for any subset $A \subset E$, we let $x(A) := \sum\limits_{e\in A} x(e)$. Also, if $y\in \R^E$ is another vector, then $y\le x$ iff $y(e)\le x(e)$ for all $e\in E$.
\begin{definition}\label{def:polymatroid}
A {\it polymatroid} in $\R^E$ is a compact non-empty subset $P$ of $\R^E_{\geq 0}$ satisfying the following properties:
\begin{itemize}
		\item[(i)] If $y \leq x \in P$, then $y \in P.$
		\item[(ii)] Given $x \in \R^E_{\geq 0}$, all maximal vectors $y \in P$ with $y \leq x$ (which are called $P$-{\it basis} of $x$) must have the same component sum $y(E)$.
\end{itemize}
\end{definition}
Let $M=(E,\cI)$ be a matroid with the rank function $r$. The {\it associated polymatroid} $P(E,r)$ of $M$ is defined by the following polytope: all $x\in\R^E$ such that
\begin{equation}\label{poly1}
		\begin{array}{ll}
			x(A) \leq r(A), \qquad \forall A \subseteq E; \\
			x \geq 0.
		\end{array}
\end{equation}
Given a matroid $M=(E,\cI)$, Edmonds \cite[Theorem 39]{edmondsgeedy} shows that the set of vertices of the associated polymatroid $P(E,r)$ described in (\ref{poly1}) is precisely the set of indicator vectors of all the independent sets in $\cI$. Furthermore, the set of all maximal vectors (with respect to the partial order $y\le x$) in $P(E,r)$ is the convex hull of all the indicator vectors for bases in $\cB$ \cite{edmondsgeedy}.
	
For simplicity, throughout this paper, we only consider loopless matroids with positive rank.
Loopless means that $r(X)=0$ implies $X=\emptyset$ and positive rank means that $r(E)>0$.
	
\subsection{Principal partition of matroids}\label{sec:pp}
In this section, we recall some results from the theory of principal partitions as presented in \cite{fujishige2009theory}. Consider a matroid $M(E,\cI)$ with the rank function $r$. We start with the following min-max relation, which relates a packing problem with an attack problem. For any positive integers $k$ and $l$, we have:
\begin{align*}
		&\max \left\{ \sum\limits_{i =1}^{k} |I_i| : \{I_i\}_{i=1}^k \subset \cI, 
		\sum_{i=1}^k \ones_{\{e\in I_i\}}\le l, \quad\forall e\in E\right\}= \min \left\{ kr(X)  + l|E-X| : X \subseteq E \right\}.
\end{align*}
For a nonnegative rational number $\lambda$, let $\cD_\lambda$ denote the set of minimizers of the submodular function $f_{\lambda}(X) = r(X) +  \la |E-X|: 2^E\rightarrow\R$. Then, $\cD_\la$ is  a distributive lattice (closed under unions and intersections). A value $\lambda$ is called {\it critical}, if $\cD_\lambda$ contains more than one element. It has been shown that the set of all critical values is finite, hence can be denoted by
\begin{equation}\label{eq:crivalue}
		0 \leq \lambda_1 < \dots < \lambda_q.
\end{equation}
	For each $i = 1,\dots, q$, define $E^{-}_{\lambda_i}$ and $E^{+}_{\lambda_i}$ as the minimum and maximum elements of $\cD_{\lambda_i}$, respectively, then $E^{-}_{\lambda_i} \subsetneq E^{+}_{\lambda_i}$. Furthermore, we have the following:
\begin{equation}
		\emptyset = E^{-}_{\lambda_1} \subsetneq E^{+}_{\lambda_1} = E^{-}_{\lambda_2} \subsetneq E^{+}_{\lambda_2} = \dots = E^{-}_{\lambda_q} \subsetneq E^{+}_{\lambda_q} = E.
\end{equation}
Moreover, each minor $(M| E^{+}_{\lambda_i}) / E^{-}_{\lambda_i}$ on $E^{+}_{\lambda_i} - E^{-}_{\lambda_i}$, with critical value $\lambda_i = k/l$,  has $k$ bases that uniformly cover every element of $E^{+}_{\lambda_i} - E^{-}_{\lambda_i}$ exactly $l$ times. In other words, as explained in the paragraph above (\ref{eq:homogeneous-matroid}), this minor is homogeneous.
	
Subsequently, the theory of principal partitions of matroids was extended to polymatroids. Let $P(E,r)$ be the polymatroid associated with the matroid $M=(E,\cI)$. Edmonds demonstrated that for any positive real $\lambda$,
\begin{equation}\label{eq:edmons}
		\max \left\{ x(E) : x \in P(E,r), x \leq \lambda \one \right\} = \min \left\{ r(X) + \lambda |E - X| : X \subseteq E \right\},
\end{equation}
where $\one$ is the column vector of all ones and the inequality  holds coordinatewise. 
	
Focusing on the left-hand side of equality (\ref{eq:edmons}), it was shown in \cite{fujishige1980lexicographically} that there exists a unique base $b^*$, called the {\it universal base} of $P(E,r)$, such that $b^* \wedge \lambda\one$ is a maximizer for each $\lambda$, where $(b^* \wedge \lambda\one)(e) := \min \{ b^*(e),\lambda \}$, for all $e \in E$. Furthermore, all critical values $\lambda_i$ as in (\ref{eq:crivalue}) are the distinct elements of $b^*$, and 
\[E^{+}_{\lambda_i} = \{ e \in E : b^*(e) \leq \lambda_i \} \quad \text{for } i = 1,\dots,q.\]
	
There are two characterizations for the universal base $ b^* $. The first one is that the universal base $ b^* $ is an optimal solution of the problem 
\begin{equation}
		\min \left\{ \sum\limits_{e \in E} x(e)^2 : x \text{ is a base in } P(E,r) \right\}.
\end{equation}
This implies that when considering the graphic polymatroid of a graph, the universal base $ b^* $ is equal to the unique optimal density of the problem (\ref{eq:co}). The second characterization is that the universal base $ b^* $ is identical to the {\it lexicographically optimal base} of $ P(E,r) $, see \cite{fujishige1980lexicographically} for further details.
	
\subsection{Packing and covering problems}
Next, we recall the base packing problem for the base family $\cB$ of a matroid $M =(E,\cI)$:
	
\begin{equation}\label{eq:packing-tree}
		{\renewcommand{\arraystretch}{1.8}
			\begin{array}{ll}
				\underset{\lambda \in \R^{\cB}_{\geq 0}}{\rm maximize} & \sum\limits_{B \in \cB} \lambda(B) \\
				\text{subject to} & \sum\limits_{B \in \cB : e \in B}\lambda(B) \leq 1, \qquad \forall e \in E.
			\end{array}
		}
\end{equation}
The covering problem for the base family $\cB$ is formulated as follows:

\begin{equation}\label{eq:base-covering-tree}
		{\renewcommand{\arraystretch}{1.8}
			\begin{array}{ll}
				\underset{\kappa \in \R^{\cB}_{\geq 0}}{\rm minimize} & \sum\limits_{B \in \cB} \kappa(B) \\
				\text{subject to} & \sum\limits_{B \in \cB : e \in B}\kappa(B) \geq 1, \qquad \forall e \in E.
			\end{array}
		}
\end{equation}
	
Using two of Edmonds's results from \cite{edmonds1965lehman}, we are able to derive the following theorem.
\begin{theorem}\label{nash}
		Let $S(M), D(M), \tau(M), \upsilon(M)$ be the strength, the fractional arboricity, the optimal values of the base packing problem, and the base covering problem of a matroid $M$, respectively. Then, we have:
		\begin{align}
			\tau(M) =  S(M), \\
			\upsilon(M) = D(M).
		\end{align}
\end{theorem}
	
\subsection{Modulus}
Let $E$ be a finite set with given weights $\si \in \R^E_{>0}$ assigned to each element $e$ in $E$. We say that $\Ga$ is a family of objects in $E$, if each object $\ga \in \Ga$ is associated to a function  $\cN(\ga,\cdot)^T: E\rightarrow \R_{\geq 0}$, which we think of as a  {\it usage vector} in $\R^E_{\geq 0}$. In other words, $\Ga$ is associated with a $|\Ga| \times |E|$ {\it usage matrix} $\cN$. From now on, we will assume that $\Ga$ is non-empty and each object $\ga\in\Ga$ uses at least one element in $E$ with a positive and finite amount.
	
A {\it density}  $\rho\in \R^E_{\geq 0}$ is a vector where $\rho(e)$ represents the {\it cost} of using the element $e \in E$. For each object $\ga \in \Ga$, we define the {\it total usage cost} of $\ga$ with respect to $\rho$
\begin{equation}\label{eq:total-usage}
\ell_{\rho}(\ga) := \sum\limits_{e \in E} \cN(\ga,e)\rho(e)= (\cN\rho)(\ga).
\end{equation}
A density $\rho \in \R^E_{\geq 0 }$ is  called {\it admissible} for $\Ga$, if for all $\ga\in \Ga$,
$ \ell_{\rho}(\ga) \geq 1 .$
In matrix notations, $\rho$ is admissible if $\cN\rho \geq \one.$
The {\it admissible set} $\Adm(\Ga)$ of $\Ga$ is defined as the set of all admissible densities for $\Ga$,
\begin{equation}\label{eq:adm-set}
		\Adm(\Ga) := \left\{ \rho \in \R^E_{\geq 0 }: \cN\rho \geq \one \right\}. 
\end{equation} 
Fix $1 \leq p < \infty$, the {\it energy} of the density  $\rho$ is defined as follows
\[ \cE_{p,\si}(\rho):=\sum\limits_{e \in E}\si(e)\rho(e)^p.\] 
When $p =\infty $,
\[\cE_{\infty,\si}(\rho):= \max\limits_{e\in E} \left\{ \si (e)\rho(e)\right\}.\]
\begin{definition}
	The {\it $p$-modulus} of $\Ga$ is
	\[\Mod_{p,\si}(\Ga):= \inf\limits_{\rho \in \Adm(\Ga)} \cE_{p,\si}(\rho).\]
\end{definition}
Equivalently, the $p$-modulus of $\Ga$ is the following optimization problem,
\begin{equation} \label{modp}
		\begin{array}{ll}
			\underset{\rho \in \R^{E}_{\geq 0}}{\text{minimize}}    &\cE_{p,\si}(\rho) \\
			\text{subject to } &\sum\limits_{e \in E} \cN(\ga,e)\rho(e) \geq 1, \quad \forall \ga \in \Ga.
			
		\end{array}
\end{equation}
When $\si$ is the vector of all ones, we omit $\si$ and write $\cE_{p}(\rho) := \cE_{p,\si}(\rho)$ and $\Mod_{p}(\Ga) :=\Mod_{p,\si}(\Ga)$.
	
\subsection{Fulkerson dual families and the $\MEO$ problem}
Let $E$ be a finite set. Let $\Ga$ be a family of objects on $E$. Let $K$ be a closed convex set in $\R^E_{\geq 0}$ such that $ \varnothing \subsetneq K \subsetneq \R^E_{\geq 0}$ and $K$ is {\it recessive}, meaning that $K +\R^E_{\geq 0} = K$. The {\it blocker} $\BL(K)$ of $K$ is defined as, \[\BL(K) = \lbr \eta \in \R^E_{\geq 0} : \eta^{T}\rho \geq 1, \forall \rho \in K \rbr. \] 
We will routinely identify $\Ga$ with the set of its usage vectors $\left\{ \cN(\ga,\cdot )^T: \ga \in \Ga \right\}$ in $\R^E_{\geq 0}$, hence we can write $\Ga \subset \R^E_{\geq 0 }$. As done above in (\ref{eq:dominant-bases}), we defined the {\it dominant} of $\Ga$ as
$ \Dom(\Ga):= \co(\Ga) + \R^E_{\geq 0},$
where $\co(\Ga)$ denotes the convex hull of $\Ga$ in $\R^E$.
	
Note that the admissible set  $\Adm(\Ga)$ defined in (\ref{eq:adm-set}) is closed, convex in $ \R^E_{\geq 0}$ and recessive. 
Next, we recall Fulkerson duality for modulus.
\begin{definition} Let $\Ga$ be a family of objects on $E$.
		The {\it Fulkerson blocker family} $\widehat{\Ga}$ of $\Ga$ is defined as the set of all the extreme points of $\Adm(\Ga)$.
		\[ \widehat{\Ga} := \Ext\left(\Adm(\Ga)\right) \subset  \R^E_{\geq 0}.\]
\end{definition}
Fulkerson blocker duality \cite{fulkersonblocking} states that
\begin{equation}\label{eq:dom-adm-block-hat}
\Dom(\widehat{\Ga})= \Adm(\Ga) = \BL(\Adm(\Gahat)),
\end{equation}
\begin{equation}\label{eq:dom-adm-block}
\Dom(\Ga)= \Adm(\widehat{\Ga}) = \BL\left(\Adm(\Ga)\right).
\end{equation}
Moreover,  $\widehat{\Ga}$  has its own Fulkerson blocker family, and
\begin{equation}\label{eq:gahathat-ga}
		\widehat{\widehat{\Ga}}  \subset \Ga .
\end{equation}
\begin{definition}\label{def:fulkerson-dual}
		Let $\Ga$ and $\widetilde{\Ga}$ be two sets of vectors in $\R^E_{\geq 0}$. We say that $\Ga$ and $\widetilde{\Ga}$ are a {\it Fulkerson dual pair} (or $\Gatil$ is a {\it Fulkerson dual family} of $\Ga$) if \[\Adm(\Gatil) = \BL(\Adm(\Ga)).\] 
\end{definition}
\begin{remark}
		If $\Gatil$ is a Fulkerson dual family of $\Ga$, then $\Ga$ is also a Fulkerson dual family of $\Gatil$ because \[\BL(\Adm(\Gatil)) = \BL(\BL(\Adm(\Ga))) = \Adm(\Ga).\]
\end{remark}
\begin{proposition}[\cite{huyfulkerson}]\label{theo:smallestF}
		Let $\Ga$ be a set of vectors in $\R^E_{\geq 0}$. Let $\Gahat$ be the Fulkerson blocker family of $\Ga$ and $\Gatil$ be a Fulkerson dual family of $\Ga$. Then, $\Gahat$ is the smallest Fulkerson dual family of $\Ga$, meaning that \[\Gahat \subset \Gatil.\] 
\end{proposition}
\begin{remark}
		When all usage vectors of $\Ga$ belong to $\left\{ 0,1 \right\} ^E$, if the support set $\lbrace e \in E: \cN(\ga,e) \neq 0 \rbrace$ of any  usage vector of $\Ga$ does not contain the support set of any other usage vector of $\Ga$, then $\Ga$ is called a {\it clutter} in the combinatorics literature. 
		An important property of a clutter $\Ga$ is $\widehat{\widehat{\Ga}} =\Ga $, see \cite{fulkersonanti}.  In this case, $\Ga$ is the Fulkerson blocker family of $\Gahat$. 
\end{remark}
When  $1<p < \infty$, let $q:=p/(p-1)$ be the Hölder conjugate exponent of $p$. For any set of weights $\si \in \R^E_{>0}$,  define the dual set of weights $\widetilde{\si}$ as $\widetilde{\si}(e):=\si(e)^{-\frac{q}{p}}$ for all $e\in E$. Let $\Gatil$ be a Fulkerson dual family of $\Ga$. Fulkerson duality for modulus  \cite[Theorem 3.7]{pietroblocking} states that
\begin{equation}
	\Mod_{p,\si}(\Ga)^{\frac{1}{p}}\Mod_{q,\widetilde{\si}}(\widetilde{\Ga})^{\frac{1}{q}}=1.
\end{equation}
Moreover, the optimal $\rho^*$ of $\Mod_{p,\si}(\Ga) $ and the optimal $\eta^*$ of $\Mod_{q,\widetilde{\si}}(\widetilde{\Ga})$ always exist, are unique, and are related as follows,
\begin{equation}\label{eq:weighted-eta-rho}
	\eta^{\ast}(e) = \frac{\si(e)\rho^{\ast}(e)^{p-1}}{\Mod_{p,\si}(\Ga)}, \quad \forall e\in E.
\end{equation}
When $p=2$, we have
\begin{equation}\label{eq:mod2}
		\Mod_{2,\si}(\Ga)\Mod_{2,\si^{-1}}(\widetilde{\Ga})=1 \qquad\text{and}\qquad   \displaystyle \eta^{\ast}(e) = \frac{\si(e)}{\Mod_{2,\si}(\Ga)}\rho^{\ast}(e) \quad \forall  e\in E.
\end{equation}
Let $\cP(\Ga)$ be the set of all probability mass functions (pmf) on $\Ga$. According to the probabilistic interpretation of modulus \cite{pietrominimal}, we can express

\begin{equation}\label{eq:modmeo}
	\Mod_2(\Ga)^{-1} = \min\limits_{\mu \in \cP(\Ga)} \mu^T\cN\cN^T\mu.
\end{equation}
 %where $\rho^*$ is the optimal density for $\Mod_2(\Ga)$.
Consider the scenario where $\Ga$ is a collection of subsets of $E$ with usage vector given by the indicator function. Given a pmf $\mu \in \cP(\Ga)$, let $\underline{\ga}$ and $\underline{\ga}'$ be two independent random objects in $\Ga$, identically distributed with law $\mu$. 
The cardinality of the overlap between $\underline{\ga}$ and $\underline{\ga}'$, is $|\underline{\ga} \cap \underline{\ga}'|$ and is a random variable whose expectation is denoted by  $\bE_\mu|\underline{\ga}\cap\underline{\ga}'|$, which equals $\mu^T\cN\cN^T\mu$. Then, the {\it minimum expected overlap} ($\MEO$) problem for $\Ga$ is formulated as $  \min\limits_{\mu \in \cP(\Ga)} \bE_\mu|\underline{\ga}\cap\underline{\ga}'|$.
%\begin{equation} \label{meo}
%{\renewcommand{\arraystretch}{1.8}
%\begin{array}{ll}
%\text{minimize}    &\bE_\mu|\underline{\ga}\cap\underline{\ga}'| \\
%\text{subject to } & \mu \in \cP(\Ga). 	 
%\end{array}}
%\end{equation}
%The optimal laws $\mu^*$ for the $\MEO$ problem are not in general unique. However, the corresponding element usage probabilities $\eta^*(e) = \bP_{\mu^*}(e \in \underline{\ga})$ do not depend on the optimal pmf $\mu^*$.  
Moreover, any pmf $\mu \in  \cP(\Ga)$ is optimal if and only if
\begin{equation}\label{eq:eta-rho}
	(\cN\mu)(e) = \rho^*(e)/\Mod_2(\Ga) \quad \forall e \in E.
\end{equation}
	
\begin{comment}
\begin{theorem}[\cite{pietrominimal}]\label{thm:meomod}
		
Let $\Ga$ be a family of objects on the ground set $E$ and let $\widetilde{\Ga}$ be a Fulkerson dual family of $\Ga$. Let $\eta^*$ be the optimal density for $\Mod_2(\widehat{\Ga})$, with $\eta^* = \cN^T\mu^*$ for any optimal pmf $\mu^*$.
Then $\rho \in  \R^E_{\geq 0}$, $\eta \in  \R^E_{\geq 0}$ and $\mu \in \cP(\Ga)$ are optimal respectively for $\Mod_2(\Ga)$, $\Mod_2(\widetilde{\Ga})$ and $\MEO(\Ga)$ if and only if the following conditions are satisfied.
\begin{itemize}
			\item[(i)] $\rho \in \Adm(\Ga), \hspace{2pt} \eta = \cN^T\mu,$
			\item[(ii)] $\eta(e) = \frac{\rho(e)}{\Mod_2(\Ga)} \quad \forall e \in E,$
			\item[(iii)] $\mu(\ga)(1-\ell_\rho(\ga)) = 0 \quad \forall \ga \in \Ga.$
\end{itemize}
In particular, 
\begin{align*}
			\MEO(\Ga) = \Mod_2(\Ga)^{-1} = \Mod_2(\widetilde{\Ga}).
\end{align*}
\end{theorem}
\end{comment}

Next, we want to recall the serial rule for the $\MEO$ problem. 
Given $A \subset E$, let $\psi_A$ be the restriction operator, 
\begin{align*}
		\psi_A: 2^E  & \rightarrow 2^A\\
		\ga \subseteq E & \mapsto \ga \cap A.
\end{align*}
Then, for each  $A \subset E$, $\psi_A$ induces a family of objects $ \psi_{A}(\Ga) = \lbrace \ga\cap A: \ga \in \Ga \rbrace.$
	
\begin{definition}\label{def:divides-gamma}
Let $\left\{ E_1,E_2 \right\} $ be a partition of the edge set $E$. For each $i = 1, 2$, we define an induced family of objects $\Ga_i :=  \psi_{E_i}(\Ga)$.
We say that a partition $\left\{ E_1,E_2 \right\} $ of the edge set $E$ {\it divides $\Ga$}, if 
$\Ga$ coincides with the {\it concatenation}
\begin{equation}\label{eq:concatenation}
\Ga_1 \oplus \Ga_2:= \left\{ \ga_1 \cup \ga_2 : \ga_i \in  \Ga_i, i =1,2 \right\}.
\end{equation}
\end{definition}
Given a partition $\left\{ E_1,E_2\right\}$ that divides $\Ga$ and a pmf $\mu \in \cP(\Ga)$. For each $i = 1, 2$, define the {\it marginal} $\mu_i \in \cP(\Ga_i)$ as follows,
\[\mu_{i}(\zeta) :=  \sum \left\{   \mu(\zeta) : \ga \in \Ga, \psi_{E_i}(\ga)= \zeta \right\} \quad \forall \zeta \in \Ga_i.\]
On the other hand, given measures $\nu_i \in \cP ( \Ga_i)$ for $i =1,2$, define their {\it product measure} in $\cP(\Ga)$ as follows, \[ \left( \nu_1  \oplus \nu_2  \right) (\ga)  := \nu_1(\zeta_1) \nu_2(\zeta_2), \]
for all $ \displaystyle \ga = \zeta_1 \cup \zeta_2$ where $ \zeta_i \in \Ga_i$, $i =1,2.$ 
	
In this case, the modulus and $\MEO$ problems split into two smaller subproblems.
\begin{theorem}[\cite{pietrofairest}]\label{thm:serialmod}
Let $\Ga$ be a family of subsets of the ground set $E$. Let $ E = E_1 \cup E_2$ be a partition that divides $\Ga$. Let $\Ga_1$ and $\Ga_2$ be the family induced by the restriction operators $\psi_{E_1}$ and  $\psi_{E_2}$. Then:
		
\bi
\item[(i)] We have \[ \MEO(\Ga) = \MEO(\Ga_1) +\MEO(\Ga_2);\]
\item[ (ii)] A pmf $\mu \in \cP(\Ga)$  is optimal for $\MEO(\Ga)$ if and only if its marginal pmfs $\mu_i \in \cP(\Ga_i), i=1,2$ are optimal for 
$\MEO(\Ga_i)$ respectively;
\item[ (iii)] Conversely, given  pmfs $\nu_i \in \cP(\Ga_i)$ that are optimal for $\MEO(\Ga_i)$  for $i=1,2$ then $\nu_1 \oplus \nu_2 $ is an optimal pmf in $\cP(\Ga) $ for $\MEO(\Ga)$; 
\item[ (iv)] For any pmf $\mu$ with marginals $\mu_i$, if $ e\in E_i$, $i =1,2$, then
\[ \bP_{\mu}(e \in \underline{\ga}) = \bP_{\mu_i}(e \in \underline{\ga_i}).\]
\ei
\end{theorem}

\section{Modulus for the base family of a matroid}\label{sec:basemodulus}
\subsection{Base modulus and the MEO problem}
Let $M=(E,\cI)$ be a loopless matroid with $r(M) >0$. Let $ \cB = \cB(M)$ be the family of all bases of  $M$ with usage vectors given by the indicator functions. We call $\cB$ the {\it base family} of $M$. Let $\widetilde{\cB}$ be a Fulkerson dual family of $\cB$ as in Definition \ref{def:fulkerson-dual}.  An explicit construction of $\widetilde{\cB}$ is the Fulkerson blocker family $\widehat{\cB}$ of $\cB$, given in  Theorem \ref{theo:fulkerson}. Let $\mu^*$ be an optimal pmf for $\MEO(\cB)$. Let $\rho^*$ and $\eta^*$ be the unique optimal densities for $\Mod_2(\cB)$ and $\Mod_2(\widetilde{\cB})$.

\begin{lemma}\label{lem:fixedsum}
	Let $\cB$ be the family of bases of a matroid $M$. Let $\mu \in \cP(\cB)$ be a probability mass function (pmf) and let $\eta = \cN^T\mu$ be the corresponding element usage probabilities. Then:
	\begin{equation}\label{eq:fixedsum}
		\eta(E) = r(E).
	\end{equation}
\end{lemma}
\begin{proof}
	Note that, by part 8 of Proposition \ref{pro:basic},  we have that $|B| = r(B) = r(E)$ for any base $B \in \cB$. Therefore:
	\[\eta(E) := \sum_{e \in E}\eta(e) = \sum\limits_{e \in E} \sum\limits_{B \in \cB}\mu(B)\cN(B,e) =  \sum\limits_{B \in \cB}\mu(B)\sum\limits_{e \in E}\cN(B,e) = r(E).\]
\end{proof}
By Lemma \ref{lem:fixedsum}, it follows that:
\begin{equation}
	\sum\limits_{e \in E} \eta^*(e) = r(E).
\end{equation}

\begin{theorem}\label{thm:meomod}
	Let $M(E ,\cI)$ be a matroid. Let $\cB$ be the base family of $M$, and let $\widetilde{\cB}$ be a Fulkerson dual family of $\cB$. Furthermore, let  $\rho \in  \R^E_{\geq 0}$, $\eta \in  \R^E_{\geq 0}$, and $\mu \in \cP(\cB)$.
	
	Then $\rho$, $\eta $ and $\mu $ are optimal, respectively, for $\Mod_2(\cB)$, $\Mod_2(\widetilde{\cB})$ and $\MEO(\cB)$ if and only if the following conditions are satisfied:
	\begin{itemize}
		\item[(i)] $\rho \in \Adm(\cB), \hspace{2pt} \eta = \cN^T\mu$;
		\item[(ii)] $\rho$ and $\eta$ are parallel, meaning that for some constant $Z>0$, $\rho(e) = Z\eta(e)$ for all $ e \in E$.
		\item[(iii)] $\mu(B)(1-\ell_\rho(B)) = 0 \quad \forall B \in \cB.$
	\end{itemize}
	In particular, the constant $Z$ in (ii) equals to $\Mod_2(\cB)$ and
	\begin{align*}
		1/Z = \MEO(\cB) = \Mod_2(\widetilde{\cB}).
	\end{align*}
\end{theorem}
\begin{proof}
	Assume that $\rho, \eta,$ and $\mu$ are optimal for $\Mod_2(\cB)$, $\Mod_2(\widetilde{\cB})$ and $\MEO(\cB)$, respectively. Then, equation  (\ref{eq:mod2}) implies that $\rho$ and $\eta$ are parallel with constant $Z = \Mod_2(\cB)$, and this establishes (ii). 
	Next, we define $\eta_{\mu} := \cN^T\mu$. Consequently, $\mu^T\cN\cN^T\mu = \eta_{\mu}^T\eta_{\mu}$ is the minimum expected overlap, and by (\ref{eq:eta-rho}), we have $\rho = \Mod_2(\cB)\eta_{\mu}$. Particularly, by (ii) with $Z$ as above, it follows that $\eta = \eta_{\mu}$. Furthermore, given that $\rho$ is optimal for $\Mod_2(\cB)$, we have $\rho \in \Adm (\cB)$, thus, (i) holds as well. Finally, part (iii) is the complementary slackness condition derived from the probabilistic interpretation for the Lagrange multipliers.
	
	Conversely, assume that $\rho, \eta,$ and $\mu$ satisfy the conditions (i), (ii), and (iii). Define \[\cB^{+}:= \lbr B \in \cB : \mu(B) > 0\rbr,\] and let $\cN^+$ be the usage matrix for $\cB^+$. To demonstrate that  $\rho$ is optimal for $\Mod_2(\cB)$, we apply the Beurling's criterion from \cite{pietrominimal}, for the case $p=2$, applied to the subfamily $\cB^+$. Given condition (iii), we have $\ell_{\rho}(B) =1$ for all $B \in \cB^+$ and by (i), $\rho \in \Adm (\cB)$. Therefore, we only need to check the second requirement in Beurling's criterion. Let's consider any  $h \in \R^E$ such that $\sum_{e \in E}h(e)\cN(B,e) \geq 0$ for every $B \in \cB^+$. Our goal is to prove that $h^T\rho \geq 0$. Indeed, 
	
	\begin{align*}
		h^T\rho &= Zh^T\eta &(\text{by } (ii))\\
		&=Zh^T\cN^T\mu &(\text{by } (i))\\
		&=Zh^T\cN^+\mu &(\text{by definition of } \cB^+)\\
		&=Z\mu^T\cN^+h \geq 0,&
	\end{align*}
since $\cN^+h \geq 0$. Therefore, we have that $\rho$ is the optimal density for $\Mod_2(\cB)$.

Next, let $ \eta^*$ and $\mu^*$ be optimal for $\Mod_2(\widetilde{\cB})$ and $\MEO(\cB)$, respectively. Then, by (\ref{eq:mod2}), it follows that $Z\eta = \rho = \Mod_2(\cB)\eta^*$. Hence, $Z\eta(E) = \Mod_2(\cB)\eta^*(E)$. Note that, by Lemma \ref{lem:fixedsum},  $\eta(E) = \eta^*(E) = r(E)$. Consequently, we have $Z = \Mod_2(\cB)$. Then, we obtain $\eta = \eta^*$. In addition, by equation (\ref{eq:eta-rho}), $\mu$ is optimal for $\MEO(\cB)$.
\end{proof}

\subsection{Beurling sets and serial rule}
First, we present some basic properties of the modulus of the base family $\cB$. Let $\widetilde{\cB}$ be a Fulkerson dual family of $\cB$ as in Definition \ref{def:fulkerson-dual}. Let $\mu^*$ be an optimal pmf for $\MEO(\cB)$. Let $\rho^*$ and $\eta^*$ be the unique optimal densities for $\Mod_2(\cB)$ and $\Mod_2(\widetilde{\cB})$.
	
\begin{definition}\label{def:fair-bases}
A base $B \in \cB$ is called a {\it fair base} if there exists an optimal pmf $\mu^*$ for $\MEO(\cB)$ such that $\mu^*(B) > 0$. The set of all fair bases is denoted by $\cB^f$.
\end{definition}
\begin{remark}\label{rem:restrict-to-fair}
	If $\cB^f\subset\cB$ is the family of fair objects in $\cB$, then we also have
	\[
	\MEO(\cB)=\MEO(\cB^f).
	\]
	This follows by restricting optimal pmf's on $\cB$ to $\cB^f$. 
\end{remark}
\begin{lemma}\label{lem:maxC}
Let $B$ be a fair base, let $x \in E - B$, and let $C$ be the unique circuit contained in $B \cup \{ x \}$ and containing $x$. Then,
\begin{equation}\label{eq:C}
	\eta^*(x) = \max\limits_{e \in C} \eta^*(e).
\end{equation}
\end{lemma}
\begin{proof}
Theorem \ref{thm:meomod} implies that (\ref{eq:C}) is equivalent to
\[\rho^*(x) = \max\limits_{e \in C} \rho^*(e).\]
Arguing by contradiction, assume that
\[\rho^*(x) < \rho^*(e^*), \qquad \text{where} \quad e^* \in \underset{e \in C}{\argmax \text{ }} \rho^*(e).\]
Since $M$ is loopless, $C$ contains at least two elements. Hence, $(C - \{ x \}) \subset B$. This implies that $e^* \in B$.
By part 3 of Proposition \ref{pro:basic}, we have $B' := (B - \{ e^* \}) \cup \{ x \}$ is a base and, recalling the total usage from (\ref{eq:total-usage}), we have
\[\ell_{\rho^*}(B') = \ell_{\rho^*}(B) - \rho^*(e^*) + \rho^*(x) < \ell_{\rho^*}(B).\]
However, since $B$ is fair, $\ell_{\rho^*}(B)=1$, by Theorem \ref{thm:meomod} (iii)
(complementary slackness). But, this contradicts the admissibility of $\rho^*$.
\end{proof}
\begin{lemma}\label{lem:positive}
We have that $\eta^*(e) > 0$ for all $e \in E$.
\end{lemma}
\begin{proof}
Assume that there exists an element $x \in E$ such that $\eta^*(x) = 0$. Let $B$ be a fair base, then all elements $e$ in $B$ have positive $\eta^*(e)$. Let $C$ be the unique circuit contained in $B \cup \{x\}$ and containing $x$, as defined in part 2 of Proposition \ref{pro:basic}. By Lemma \ref{lem:maxC}, this leads to a contradiction.
\end{proof}

\begin{definition}\label{def:homogeneous}
Given a matroid $M(E,\cI)$, let $\cB$ be the family of bases of $M$, and let $\widetilde{\cB}$ be a Fulkerson dual family of $\cB$. Let $\rho^*$ and $\eta^*$ be the unique optimal densities for $\Mod_2(\cB)$ and $\Mod_2(\widetilde{\cB})$, respectively. Then, $M$ is said to be {\it homogeneous} if $\eta^*$ is constant, or equivalently, $\rho^*$ is constant.
\end{definition}
\begin{remark}
Later, we will show that Definition \ref{def:homogeneous} is equivalent to the concept of homogeneous matroids mentioned in the introduction. From now on, we will use Definition \ref{def:homogeneous} when discussing homogeneous matroids.
\end{remark}
	
\begin{theorem}\label{theo:etaadm}
Let $\cB$ be the base family of a matroid $M$. Let $\widetilde{\cB}$ be a Fulkerson dual family of $\cB$. Let $\eta^*$ be the optimal density for $\Mod_2(\widetilde{\cB})$. Define the density
\begin{equation}
\eta_{hom}(e) := \frac{r(E)}{|E|} \quad \forall e \in E.
\end{equation}
Then, \[ \Mod_2(\widetilde{\cB}) \geq \cE_2(\eta_{hom}) = \frac{r(E)^2}{|E|}.\]
Moreover, $M$ is homogeneous if and only if $\eta_{hom} \in \Adm(\widetilde{\cB})$.
\end{theorem}
\begin{remark}\label{rem:constant-el-prob}
Let $\mu\in\cP(\cB)$ be a pmf, and let $\eta=\cN^T\mu$ be the corresponding element usage probabilities. If $\eta$ is constant, then $\eta$ is optimal for $\Mod_2(\widetilde{\cB})$, because such $\eta$ is automatically admissible for $\widetilde{\cB}$. Indeed, $\eta\in \Dom (\cB)$. Hence by (\ref{eq:dom-adm-block}) and Definition \ref{def:fulkerson-dual}, $\eta\in\Adm (\widetilde{\cB})$.
\end{remark}
\begin{proof}
First, define the expectation and variance of a vector $\xi \in \R^E$ as:
\[\bE(\xi) := \frac{1}{|E|} \sum\limits_{e \in E} \xi(e),\qquad\text{and}\qquad\Var(\xi) := \bE(\xi^2) - (\bE(\xi))^2,\]
where the square is taken element-wise. Then, we have that
\begin{align*}
\cE_2(\eta) = \sum\limits_{e \in E} \eta(e)^2 &= |E| (\Var(\eta) + (\bE(\xi))^2) = |E| \left(\Var(\eta) + \frac{r(E)^2}{|E|^2}\right)\\
& \geq 0 +  \frac{r(E)^2}{|E|} =  \frac{r(E)^2}{|E|},
\end{align*}
where equality holds if and only if $\eta$ is constant. 
		
If $M$ is homogeneous, then $\eta^* = \frac{r(E)}{|E|} = \eta_{hom}$. If $\eta_{hom}$ is admissible for $\widetilde{\cB}$, then $\Mod_2(\widetilde{\cB})$ achieves its minimum $\frac{r(E)^2}{|E|}$ at $\eta^* = \eta_{hom}$.
\end{proof}
	
To gain a better understanding of $\eta^*$, we consider the following lemma:

\begin{lemma}\label{lem:BcapX}
Let $\cB$ be the base family of a matroid $M$, and let $B$ be a base in $\cB$. For any subset $X \subseteq E$, it holds that:
\begin{equation}\label{eq:BcapX}
r(X) \geq |B \cap X| \geq r(E) - r(E - X).
\end{equation}
\end{lemma}
	
\begin{proof}
Since $B$ is a base, by the hereditary property in Definition \ref{def:independent-set} (I2), both $B \cap X$ and $B - X$ are independent sets. Then, we have:
\begin{align*}
r(E - X) &\geq r(B - X)  & (\text{by Proposition \ref{pro:basic} (5)})\\
& = |B - X| & (\text{by Proposition \ref{pro:basic} (7)})\\
& = |B| - |B \cap X| \\
& = r(E) - |B \cap X| & (\text{by Proposition \ref{pro:basic} (8)}).
\end{align*}
and
\begin{align*}
r(X) &\geq r(B \cap X)   & (\text{by Proposition \ref{pro:basic} (5)})\\
& = |B \cap X| & (\text{by Proposition \ref{pro:basic} (7)}).\\
\end{align*}
	\end{proof}

The following lemma gives a relation between $\eta^*$ and the set of fair bases $\cB^{f}$. 

\begin{lemma}\label{lem:eta_tightset}
Given a matroid $M(E,\cI)$, let $\cB$ be the base family of $M$. Let $\widetilde{\cB}$ be a Fulkerson dual family of $\cB$. Let $\eta^*$ be the optimal density for $\Mod_2(\widetilde{\cB})$. Then, 
\bi
\item[(i)]		For any subset $X \subseteq E$, we have:
\begin{equation}\label{eq:etaX}
\eta^*(X) \geq r(E) - r(E - X).
\end{equation}
Moreover, equality holds in (\ref{eq:etaX}), if and only if, every fair base $B \in \cB^f$ satisfies $r(E - X) = |B - X|$.
\item[(ii)]		
For any subset $Y \subseteq E$, we have:
\begin{equation}\label{eq:etaY}
r(Y) \geq \eta^*(Y).
\end{equation}
Moreover, equality holds in (\ref{eq:etaY}), if and only if, every fair base $B \in \cB^f$ satisfies $r(Y) = |B \cap Y|$.
\ei	\end{lemma}
Let us introduce the notions of {\it Beurling sets} and {\it complement-Beurling sets} in the following definition. The term `Beurling' is inspired by Beurling’s Criterion and the notion of Beurling families in the theory of modulus; see \cite[Theorem 2.1]{pietrominimal}.
\begin{definition}\label{def:beurling-set}
The set $X$ is said to be a {\it Beurling set}, if equality holds in (\ref{eq:etaX}).
The set $Y$ is said to be a  {\it complement-Beurling set}, if equality holds in (\ref{eq:etaY}).
\end{definition}
\begin{remark}\label{rem:complement-beurling}
Note that, since $r(E)=\eta^*(E)$, a set $X$ is a Beurling set if and only if $E - X$ is a complement-Beurling set.
\end{remark}

\begin{proof}[Proof of Lemma \ref{lem:eta_tightset}]
Write $\eta^* = \cN^T\mu^*$ for some pmf $\mu^*$ optimal for the $\MEO$ problem, as in Theorem \ref{thm:meomod}. Then,
		\begin{align*}
\eta^*(X) &= \sum_{e \in X}\eta^*(e) = \sum\limits_{e \in X} \sum\limits_{B \in \cB}\mu^*(B)\cN(B,e) \\
&= \sum\limits_{B \in \cB}\mu^*(B)\sum\limits_{e \in X}\cN(B,e) = \sum\limits_{B \in \cB}\mu^*(B)|B \cap X|\\
& \ge  \sum\limits_{B \in \cB}\mu^*(B)(r(E)-r(E-X)),\\
		\end{align*}
where the last line follows by applying the second inequality in (\ref{eq:BcapX}). Finally, to complete the proof of part (i), we note that $\sum_{B \in \cB}\mu^*(B)=1$.

Moreover, the inequality in the last line holds as equality if and only if the second inequality in (\ref{eq:BcapX}) holds as equality for every fair base, see Definition \ref{def:fair-bases}.

To prove part (ii), let $X$ be defined so that $Y=E-X$, then use the same idea as in Remark \ref{rem:complement-beurling}.
\end{proof}
	
We recall that a set $X \subseteq E$ is said to be closed if $\cl(X) = X$. Now, we define a {\it complement-closed set} as follows.
	
\begin{definition}
A set $X \subseteq E$ is said to be complement-closed if $\cl(E - X) = E - X$. 
\end{definition}
	
\begin{lemma}\label{lem:bcc}
Let $X \subseteq E$. If $X$ is a Beurling set, then $X$ is a complement-closed set.
\end{lemma}
	
\begin{proof}
Assume that $\cl(E - X) \supsetneq E - X$. Then, there exists an element $x \in X$ such that $x \in \cl(E - X)$. By definition of the closure operator, in (\ref{eq:closure-operator}), $r((E - X) \cup \{x\}) = r(E - X)$.  Lemma \ref{lem:eta_tightset} implies
\begin{align*}
\eta^*(X) & > \eta^*(X) - \eta^*(x) = \eta^*(X - \{x\})  &\text{(by Lemma \ref{lem:positive})}\\
& \geq r(E) - r((E - X) \cup \{x\})  &\text{(by Lemma \ref{lem:eta_tightset})}\\
&= r(E) - r(E - X)   &\text{(by (\ref{eq:closure-operator}))}\\
&= \eta^*(X).   & \text{(by Definition \ref{def:beurling-set})}
\end{align*}
This results in a contradiction.
\end{proof}

Next, we aim to apply the serial rule for base modulus. To that end, we introduce the following theorem. Recall that the symbol $\oplus$ for concatenation is defined in (\ref{eq:concatenation}).
	
\begin{theorem}\label{thm:baseserial}
Given a matroid $M(E,\cI)$, let $\cB = \cB(M)$ be the base family of $M$. Let $X \subseteq E$, let $M \setminus X$ be the deletion of $X$ from $M$. Let $M / (E - X)$ be the contraction of $E - X$ in $M$. Let $\cB^X$ be the set of bases $B \in \cB$ which satisfies $r(E - X) = |B - X|$. Then,
\begin{equation}
\cB(M \setminus X) \oplus \cB(M / (E - X)) = \cB^X.
\end{equation} 
\end{theorem}
	
\begin{proof}
First, we aim to show that
\begin{equation}\label{eq:proof1}
\cB(M \setminus X) \oplus \cB(M / (E - X)) \subset \cB^X.
\end{equation}
By part 3 of Proposition \ref{pro:contraction}, we have
\[\cB(M \setminus X) \oplus \cB(M / (E - X)) \subset \cB(M).\]
Let $B_1 \in \cB(M \setminus X), B_2 \in \cB(M / (E - X))$, and let $B_3 = B_1 \cup B_2$. Then, we have
\begin{align*}
|B_3 - X| & = |B_1| &\text{(by construction)}\\
&= r_{M \setminus X}(B_1) &\text{(by definition of $B_1$)} \\
&= r_{M \setminus X}(E - X) & \text{(by part 8 of Proposition \ref{pro:basic})}\\
& = r_M(E - X). &\text{(by part 3 of Proposition \ref{pro:deletion})}
\end{align*}
Hence, (\ref{eq:proof1}) holds.
		
Next, we aim to show
\begin{equation}\label{eq:proof2}
			\cB(M \setminus X) \oplus \cB(M / (E - X)) \supset \cB^X.
\end{equation}
Let $B \in \cB^X$, then $B = (B - X) \cup (B \cap X)$. 

On one hand, by definition of  $\cB^X$ and part 3 of Proposition \ref{pro:deletion}, we have
\begin{equation}\label{eq:rankde}
			|B - X| = r_M(E - X) = r_{M \setminus X}(E - X).
\end{equation} 
Note that $B - X\subset B$, so $B-X$ is an independent set of $M$. By part 1 of Proposition \ref{pro:deletion}, $B - X$ is independent in $M \setminus X$. Finally, by part 8 of Proposition \ref{pro:basic} and (\ref{eq:rankde}), we have
\[(B - X) \in \cB(M \setminus X).\]
		
On the other hand, to show that
\[B \cap X \in \cB(M / (E - X)),\]
note that $(B \cap X) \cup (B - X) = B \in \cB(M)$, and $B - X \in \cB(M \setminus X)$, hence it satisfies the conditions in part 3 of Proposition \ref{pro:contraction}.
Therefore, (\ref{eq:proof2}) holds and the proof is completed.
\end{proof}
\begin{remark}
Note that in Theorem \ref{thm:baseserial}, $X$ is not necessarily closed. This will be useful later.
\end{remark}
Now we can state the serial rule for base modulus using the Beurling set property of Lemma \ref{lem:eta_tightset}, the serial rule in Theorem \ref{thm:serialmod} and the concatenation property in Theorem \ref{thm:baseserial}.
	
\begin{theorem}\label{thm:serialmodmod}
Given a matroid $M(E,\cI)$, let $\cB = \cB(M)$ be the base family of $M$, and let $\widetilde{\cB}$ be a Fulkerson dual family of $\cB$. Let $X \subseteq E$ be a Beurling set, and let $M \setminus X$ be the deletion of $X$ from the matroid $M$. Let $M / (E - X)$ be the contraction of $E -X$ in $M$.
Let $\eta^*$ be the optimal density for $\Mod_2(\widetilde{\cB})$. Then, the following hold:
\bi
\item[ (i)] The minimum expected overlap splits as 
 \begin{equation}
\MEO (\cB(M)) = \MEO (\cB(M \setminus X)) + \MEO (\cB(M /(E-X)));
\end{equation}
\item[ (ii)] A pmf $\mu \in \cP(\cB(M))$  is optimal for $\MEO(\cB(M))$ if and only if its marginal pmf $\mu_{\cB(M \setminus X)} \in \cP(\cB(M \setminus X))$ is optimal for 
$\MEO(\cB(M \setminus X))$  and its marginal pmf $\mu_{\cB(M /(E-X))} \in \cP(\cB(M /(E-X)))$ is optimal for 
$\MEO(\cB(M /(E-X)))$;
\item[ (iii)] Conversely, given  a pmf $\nu_1 \in \cP(\cB(M \setminus X))$ that is optimal for $\MEO(\cB(M \setminus X))$ and a pmf $\nu_2 \in \cP(\cB(M /(E-X) ))$ that is optimal for  $\MEO(\cB(M /(E-X)))$, then $\nu_1 \oplus \nu_2 $ is an optimal pmf in $\cP(\cB(M)) $ for $\MEO(\cB(M))$; 
\item[ (iv)] The restriction of $\eta^*$ onto $E-X$ is optimal for $\Mod_2(\widetilde{\cB}(M \setminus X)) $ and the  restriction of $\eta^*$ onto $X$  is optimal for $\Mod_2(\widetilde{\cB}(M/(E-X)))$. 
\ei
\end{theorem}

\begin{proof}
We apply Theorem \ref{thm:serialmod} for the families $\Ga_1=\cB(M \setminus X) $ and $\Ga_2=\cB(M / (E - X)) $.

By Theorem \ref{thm:baseserial}, the concatenation of $\Ga_1$ and $\Ga_2$ in this case is given by 
\begin{equation*}
\cB(M \setminus X) \oplus \cB(M / (E - X)) = \cB^X,
\end{equation*} 
where $\cB^X$ is the set of bases $B \in \cB$ which satisfy $r(E - X) = |B - X|$. 

Furthermore, by Lemma \ref{lem:eta_tightset}, since $X$ is a Beurling set, it follows that $\cB^f \subset \cB^X$. Consequently, by Remark \ref{rem:restrict-to-fair}, $\MEO(\cB(M)) = \MEO(\cB^X)$. The rest of the proof follows by applying Theorem \ref{thm:serialmod}.
\end{proof}
	
\section{Strength and Fractional Arboricity}\label{sec:strength}
In this section, we demonstrate that base modulus recovers known results in the theory of principal partition of matroids, see \cite{catlin1992,fujishige2009theory}.
	
\begin{theorem}\label{thm:max}
Given a loopless matroid $M=(E,\cI)$, let $\cB$ be the base family of $M$, and let $\widetilde{\cB}$ be a Fulkerson dual family of $\cB$. Let $\mu^*$ be an optimal pmf for $\MEO(\cB)$, and let $\eta^* = \cN^T\mu^*$ be the optimal density for $\Mod_2(\widetilde{\cB})$. Let $S(M)$ be the strength of $M$. Define:
\begin{equation}\label{eq:max}
X := E_{\max} := \left\{ e \in E : \eta^*(e) = \max\limits_{e' \in E} \eta^*(e') =: \eta^*_{\max} \right\}.
\end{equation}
Then, we have the following properties:
\begin{itemize}
\item[(1)] $X$ is a Beurling set.
\item[(2)] $\cl(E - X) = E - X$.
\item[(3)] \[\eta^*_{\max} = \frac{r(E) - r(E - X)}{|X|}.\]
\item[(4)] \[\eta^*_{\max} = \frac{1}{S(M)}.\]
\item[(5)] The matroid $M / (E - E_{\max})$ is homogeneous.
\end{itemize}
\end{theorem}
\begin{proof}
For part 1, Lemma \ref{lem:eta_tightset} (i) indicates that the statement is equivalent to $r(E - X) = |B - X|$ for any fair base $B \in \cB^f$. Suppose $r(E - X) > |B - X|$ for some fair base $B$. Let $S$ be a maximal independent set in $E-X$, then $|S| = r(E-X)  > |B - X|$. Note that $B - X$ and $S$ are independent and $(B - X) \subset (E - X)$. By the exchange property (I3) of Definition \ref{def:independent-set}, there exists $y \in (S - (B-X)) \subset ((E - X) - (B - X))$ such that $(B - X) \cup \{y\}$ is independent.
Since $B$ is a base and $y \notin B$, let $C$ be the unique circuit within $B \cup \{y\}$, containing $y$. Since $(B - X) \cup \{y\}$ is independent, we have that $C \not\subseteq (B - X) \cup \{y\}$. Consequently, $C \cap X \neq \emptyset$. Let $x$ be an element in $C \cap X$. By Lemma \ref{lem:maxC}, we have $\eta^*(y) \geq \eta^*(x)=\eta^*_{max}$. This is a contradiction because $y \in S \subset E- X$. Therefore, $X$ is a Beurling set, and part 1 is proved.

Part 2 follows from Lemma \ref{lem:bcc}. 

Furthermore, part 3 holds because
\begin{align*}
\eta^*_{max} & = \frac{\eta^*(X)}{|X|} &\text{(by definition of $X$)}\\
& = \frac{r(E) - r(E - X)}{|X|}. &\text{(since $X$ is a Beurling set)}
 \end{align*}

For part 4, note that for any set $\emptyset \neq Y \subseteq E$, we have
\begin{align*}
\eta^*_{max}&\geq \frac{\eta^*(Y)}{|Y|} & (\text{by definition of $\eta^*_{max}$})\\
&\geq \frac{r(E) - r(E - Y)}{|Y|}. & (\text{by Lemma \ref{lem:eta_tightset} (i)})
\end{align*}
This implies that, for any set $Y \subseteq E$ with $r(E) > r(E - Y)$, we can write
\[\frac{1}{\eta^*_{max}} \leq \frac{|Y|}{r(E) - r(E - Y)}.\]
Thus, by definition of the strength problem in (\ref{eq:strength-problem}), we have
\[
\frac{1}{\eta^*_{max}} \leq S(M).
\]
Moreover, equality holds by part 3, and the minimum is achieved by the set $X$.
	
Finally, by Theorem \ref{thm:serialmodmod} (iv), the restriction of $\eta^*$ to $E_{max}$ is optimal for $\Mod_2(\tilde{\cB}(M/(E-E_{max}))$. Also, $\eta^*$ is constant on $E_{max}$. Therefore, part 5 holds.
	
\end{proof}
	
\begin{theorem}\label{thm:min}
Given a loopless matroid $M = (E, \cI)$, let $\cB$ be the base family of $M$, and let $\widetilde{\cB}$ be a Fulkerson dual family of $\cB$. Let $\mu^*$ be an optimal pmf for the $\MEO(\cB)$ problem, and let $\eta^* = \cN^T \mu^*$ be the optimal density for $\Mod_2(\widetilde{\cB})$. We define
\begin{equation}\label{eq:min}
Y := E_{min} := \left\{ e \in E : \eta^*(e) = \min_{e' \in E} \eta^*(e') =: \eta^*_{min} \right\}.
\end{equation}
Then,
\begin{enumerate}
\item $Y$ is a complement-Beurling set.
\item $\cl(Y) = Y$.
\item \[\eta^*_{min} = \frac{r(Y)}{|Y|} > 0.\]
\item \[\eta^*_{min} = \frac{1}{D(M)}.\]
\item The matroid $M \setminus (E - E_{min})$ is homogeneous.
\end{enumerate}
\end{theorem}
	
\begin{proof}
For part 1, by Lemma \ref{lem:eta_tightset} (ii), the statement is equivalent to $|B \cap Y| = r(Y)$ for any fair base $B \in \cB^f$. By Lemma \ref{lem:BcapX}, we have $|B \cap Y| = r(B \cap Y) \leq r(Y)$ for any base $B$. Suppose that $|B \cap Y| < r(Y)$ for some fair base $B$. 
Let $T$ be a maximal independent set in $Y$, then $|T| = r(Y)  > |B \cap Y|$. Note that  $B \cap Y$  and $T$ are independent and $(B \cap Y) \subset Y$. By the exchange property (I3) of Definition \ref{def:independent-set}, there exists $z \in (T - (B \cap Y)) = T - B$ such that $(B \cap Y) \cup \{z\}$ is independent.
Since $B$ is a base and $z \notin B$, let $C$ be the unique circuit within $B \cup \{z\}$, containing $z$. Since $(B \cap Y) \cup \{z\}$ is independent, it follows that $C \not\subseteq (B \cap Y) \cup \{z\}$. Consequently, $C - Y \neq \emptyset$. Let $x$ be an element in $C - Y$. By Lemma \ref{lem:maxC}, we have $ \eta^*_{min} = \eta^*(z) \geq \eta^*(x)$. This leads to a contradiction as $x \notin Y$. Therefore, $Y$ is a complement-Beurling set.
			
Part 2 follows from by Lemma \ref{lem:bcc}. 

Since $Y$ is a complement-Beurling set, we have
\begin{align*}
\eta^*_{min} &= \frac{\eta^*(Y)}{|Y|} &	(\text{by  definition of $Y$})\\
& = \frac{r(Y)}{|Y|}  & (\text{by the definition of complement-Beurling sets})\\
& >0, &(\text{$Y$ is nonempty and $M$ is loopless})
\end{align*}
and this proves part 3.
		
For part 4, consider any nonempty set $Z \subseteq E$, we have 
\begin{align*}
\eta^*_{min} &\leq \frac{\eta^*(Z)}{|Z|} & (\text{by definition of $\eta^*_{min}$})\\
&\leq \frac{r(Z)}{|Z|}. & (\text{by Lemma \ref{lem:eta_tightset} (ii)})
\end{align*}
This implies that for any nonempty set $Z \subseteq E$ with $r(Z) > 0$, we have
\[\frac{1}{\eta^*_{min}} \geq \frac{|Z|}{r(Z)}.\]
Therefore, the fractional arboricity problem reaches its maximum at the set $Y$.
		
Part 5 follows from Theorem \ref{thm:serialmodmod} (iv).
\end{proof}

Next, we show  that Theorems \ref{thm:max} and \ref{thm:min} can be used to recover known characterizations of the set of matroids $M$ for which $S(M) = D(M)$, as given in \cite[Theorem 6]{catlin1992}. 
	
\begin{corollary}\label{coro:homogeneous} Given a matroid $M(E,\cI)$ with rank function $r$.  Let $S(M)$ be the strength of $M$, $D(M)$ be the fractional arboricity of $M$, and $\theta(M)$ be the density of $M$. Then, the following statements are equivalent:
\bi 
	\item[ (i)]  $M$ is homogeneous.
	\item[ (ii)]   $S(M)=\theta(M).$
	\item[ (iii)]   $D(M) =\theta(M)$.
	\item[ (iv)]   $S(M)=D(M)$.
		
\ei
\end{corollary}
	
\begin{proof}
By Definition \ref{def:homogeneous}, $M$ is homogeneous if and only if $\eta^*_{max} = \eta^*_{min}$. By Theorems \ref{thm:max} and \ref{thm:min} and definitions of $S(M)$,  $D(M)$, $\theta(M)$, we have
\begin{equation}\label{eq:inequalities-sequence}
\frac{1}{\eta^*_{max}}  =S(M) \leq \theta(M) \leq D(M) = \frac{1}{\eta^*_{min}}.
\end{equation}
Therefore, it is enough to show that (ii) implies (i) and (iii) implies (i). 
		
Assume that we have (ii), then
\begin{align*}
\eta^*_{max} &\geq \frac{\eta^*(E)}{|E|}& (\text{by definition of $\eta^*_{max}$})\\
&  \geq \frac{r(E) - r(E \setminus E)}{|E|} & (\text{by Lemma \ref{lem:eta_tightset} (i)})\\
&= \eta^*_{max}. & (\text{by the assumption})	
\end{align*}
This implies that equality holds throughout, hence $\eta^*$ is  constant, because its average is equal to its max. 

Finally, assume that we have (iii), then we obtain 
\begin{align*}
	\eta^*_{min} &\leq \frac{\eta^*(E)}{|E|}& (\text{by definition of $\eta^*_{min}$})\\
	&  \leq \frac{r(E)}{|E|} & (\text{by Lemma \ref{lem:eta_tightset} (ii)})\\
	&= \eta^*_{min}. & (\text{by the assumption})	
\end{align*}
This implies that equality holds throughout and $\eta^*$ is a constant. 

\end{proof}

\section{Base packing and base covering problems}\label{sec:basepacking}	
	
In this section, we present an alternative proof of Theorem \ref{nash}, using the theory of base modulus. To lay the groundwork for this proof, we provide several lemmas.

Given a matroid $M(E,\cI)$, let $\cB = \cB(M)$ be the base family of $M$, and let $\widetilde{\cB}$ be a Fulkerson dual family of $\cB$. Let $\eta^* = \cN^T\mu^*$ be the optimal density for $\Mod_2(\widetilde{\cB})$.
Let $S(M), D(M),\theta(M), \tau(M),$ and $\upsilon(M)$ be the strength, the fractional arboricity, the density, the optimal values of the base packing problem (\ref{eq:packing-tree}), and the base covering problem (\ref{eq:base-covering-tree}) of $M$, respectively.
	
\begin{lemma}
\label{lem:5inequalities}
We have the following chain of inequalities:
\begin{equation}\label{eq:longone}
\frac{1}{\eta^*_{max}} = S(M) \leq \tau(M) \leq \theta(M) \leq \upsilon(M) \leq D(M) = \frac{1}{\eta^*_{min}}.
\end{equation}
\end{lemma}
\begin{proof} 
		
Note that in (\ref{eq:inequalities-sequence}) we have already established the two equalities. Therefore, it remains to show the four inequalities.
		
\begin{enumerate}
\item We aim to show that
\[\frac{1}{\eta^*_{max}} \leq \tau(M).\]
Let $\mu^*$ be an optimal probability mass function (pmf) for $\MEO(\cB)$. We will show that the density $\la=\mu^*/\eta^*_{max} \in \R^{\cB}_{\geq 0}$ satisfies all the constraints of the base packing problem (\ref{eq:packing-tree}). Indeed,  for a fixed $e\in E$, using Theorem \ref{thm:meomod} (i), we see that
\[\sum_{ B \in \cB: e \in B} \frac{\mu^*(B)}{\eta^*_{max}} = \frac{\eta^*(e)}{\eta^*_{max}} \leq 1.\]
Therefore,
\[\tau(M) \geq \sum_{B \in \cB} \frac{\mu^*(B)}{\eta^*_{max}} = \frac{1}{\eta^*_{max}}.\]
\item 
Likewise, we want to show that

\[\upsilon(M) \leq \frac{1}{\eta^*_{min}}.\]

Let $\mu^*$ be an optimal probability mass function (pmf) for $\MEO(\cB)$. Following the approach in part 1, we show that the density $\kappa=\mu^*/\eta^*_{min} \in \R^{\cB}_{\geq 0}$ satisfies all the constraints of the base covering problem (\ref{eq:base-covering-tree}). Again, for a fixed $e\in E$, by Theorem \ref{thm:meomod} (i), we have
\[\sum_{ B \in \cB: e \in B} \frac{\mu^*(B)}{\eta^*_{min}} = \frac{\eta^*(e)}{\eta^*_{\min}} \geq 1.\]
Therefore,
\[\upsilon(M) \leq \sum_{B \in \cB} \frac{\mu^*(B)}{\eta^*_{min}} = \frac{1}{\eta^*_{min}}.\]

\item Next, we aim to show that
\[\tau(M) \leq \theta(M).\]
Let $\lambda^*$ be an optimal solution for the base packing problem (\ref{eq:packing-tree}). We have
\begin{align}\label{eq:PV-THETA}
\tau(M)r(E) & = \sum_{B \in \cB} r(E)\lambda^*(B) & \text{(by definition of $\la^*$)}\notag\\
& =  \sum_{B \in \cB} \sum_{e \in E}\lambda^*(B)\cN(B,e)  & \text{(since $r(E)=|B|$)} \notag\\
& = \sum_{e \in E}\sum_{B \in \cB}\lambda^*(B)\cN(B,e)  \leq \sum_{e \in E} 1 = |E|,&\text{(by constraint in (\ref{eq:packing-tree}))}
\end{align}
where equality holds if and only if 
\[\sum_{B \in \cB}\lambda^*(B) \cN(B,e) = 1, \quad \forall e \in E.\]
Consequently, we conclude that $\tau(M) \leq |E| / r(E)=\theta(M)$.
\item Lastly, we establish that
\[\theta(M) \leq \upsilon(M).\]
Let $\kappa^* \in \R^{\cB}_{\geq 0 }$ be an optimal solution for the base covering problem (\ref{eq:base-covering-tree}). We derive
\begin{align}\label{eq:CV-THETA}
\upsilon(M)r(E) & = \sum_{B \in \cB} r(E)\kappa^*(B)   & \text{(by definition of $\kappa^*$)}\notag\\
& =  \sum_{B \in \cB} \sum_{e \in E} \kappa^*(B)\cN(B,e)& \text{(since $r(E)=|B|$)} \notag\\
& =  \sum_{e \in E}\sum_{ B \in \cB} \kappa^*(B)\cN(B,e) \geq \sum_{e \in E} 1 = |E|,&\text{(by constraint in (\ref{eq:base-covering-tree}))}
\end{align}
where the equality holds if and only if 
\[\sum_{B \in \cB} \kappa^*(B)\cN(B,e) = 1, \quad \forall e \in E.\]
Consequently, this implies that $\upsilon(M) \geq |E| / r(E)=\theta(M)$.
\end{enumerate}
\end{proof}
\begin{remark}\label{re:homo}
 Based on the chain of inequalities (\ref{eq:longone}), if $\eta^*$ is constant, then it follows that $\tau(M) = \theta(M) = \upsilon(M)$.
\end{remark}
The next corollary gives a converse for Remark \ref{re:homo}.
\begin{corollary}\label{coro:eta1}
The following statements hold:
\begin{enumerate}
\item If $\tau(M) = \theta(M)$, then $\eta^*$ is constant.
\item If $\upsilon(M) = \theta(M)$, then $\eta^*$ is constant.
\end{enumerate}
\end{corollary}
\begin{proof}
Assume that $\tau(M) = \theta(M)$. Let $\lambda^*$ be an optimal solution for the base packing problem. Consider the pmf defined as
\[\mu^* := \frac{\lambda^*}{\sum\limits_{B \in \cB} \lambda^*(B)}.\] 
Then, the inequality in (\ref{eq:PV-THETA}) holds as equality. Consequently, the pmf $\mu^*$ induces constant element usage probabilities $\eta := \cN^T\mu^*$. By Remark \ref{rem:constant-el-prob}, it follows that $\eta^* = \eta$.
		
The proof for the part 2. follows a similar approach.
\end{proof}

Next, we prove some helpful properties of deletions and contractions.

\begin{lemma}\label{lem:5sd}
Given a matroid $ M(E,\cI) $ and subsets $ X, Y \subseteq E$, the following properties hold:
\begin{enumerate}
\item If $ \cl(E - X) \neq E $, then $ S(M) \leq S(M / (E - X)) $.
\item If $ Y \neq \emptyset $, then $ D(M) \geq D(M \setminus (E - Y)) $.
\item If $ \cl(E - X) \neq E $, then $ \tau(M) \leq \tau(M / (E - X)) $.
\item If $ Y \neq \emptyset $, then $ \upsilon(M) \geq \upsilon(M \setminus (E - Y)) $.
\end{enumerate}
\end{lemma}
	
\begin{proof}
We prove each part in the order given.
\begin{enumerate}
\item
We have
\begin{align*}
S(M) &= \min \left\{ \frac{|E - T|}{r(E) - r(T)} : T \subseteq E, \cl(T) \subsetneq E \right\} \\
&\leq \min \left\{ \frac{|E - T|}{r(E) - r(T)} : (E - X) \subseteq T \subseteq E, \cl(T) \subsetneq E \right\}.
\end{align*}
Since $E-T\subset X$, we can rewrite the right-hand side as
\begin{align*}
&  \min \left\{ \frac{|X| - |T \cap X|}{[r(E) - r(E - X)] - [r(T) - r(E - X)]} : (E - X) \subseteq T \subseteq E, \cl(T) \subsetneq E \right\}.\\
&= \min \left\{ \frac{|X| - |T - (E - X)|}{[r(E) - r(E - X)] - [r(T) - r(E - X)]} : (E - X) \subseteq T \subseteq E, \cl(T) \subsetneq E \right\}.
\end{align*}
By part 4 of Proposition \ref{pro:contraction}, we have
\[ r_{M / (E - X)}(X) = r_{M}(E) - r_{M}(E - X), \]
and
\[ r_{M / (E - X)}(T - (E - X)) = r_{M}(T) - r_{M}(E - X). \]

Therefore, 
\begin{align*}
	 S(M) &\leq \min \left\{ \frac{|X| - |T - (E - X)|}{r_{M / (E - X)}(X) - r_{M / (E - X)}(T - (E - X))} : (E - X) \subseteq T \subseteq E, \cl(T) \subsetneq E \right\}\\
	  &= S(M / (E - X)).
\end{align*}	
		% ... (other parts of the proof)
		
\item We have 
\begin{align*}
D(M) &= \max \lbr \frac{|T|}{r(T)} : T\subseteq E, r(T)>0\rbr \\
& \geq \max \lbr \frac{|T|}{r(T)} : T\subseteq Y, r(T)>0\rbr. 
\end{align*}
Thus, by part 3 of Proposition \ref{pro:deletion}, the right hand-side equals
\begin{align*}
&  \max \lbr \frac{|T|}{r_{M \setminus (E-Y)}(T)} : T\subseteq Y, r(T)>0\rbr \\
& = D(M \setminus (E-Y)).
\end{align*}

\item First, we construct a map $f$ from $ \cB(M)$ to $\cB(M/(E-X))$ such that $f(B) \subset B$ for any base $B \in \cB(M)$. Let $B \in \cB(M)$, by part 4 of Proposition \ref{pro:contraction}, we have 
\begin{align*}
	 r_{M / (E-X)}(B-(E-X)) &= r_{M}(B) - r_{M}(E-X) \\
	 & = r_{M}(E) - r_{M}(E-X) \\
	 & = r_{M / (E-X)}(X).
\end{align*}
Therefore,  there exists a base $B_0 \in \cB(M/(E-X))$ such that $B_0 \subset B-(E-X) = B \cap X \subset B$. Then, we define $f(B) := B_0$. If $B_0$ is not unique, we pick one randomly.
		
Next, let $\lambda^*$ be an optimal solution for the base packing problem (\ref{eq:packing-tree}) of $M$. We define a vector $\omega \in \R^{\cB(M/(E-X))}_{\geq 0}$ as follows:
		
\[ \omega(\varsigma) := \sum\limits_{B \in \cB(M)} \lambda^*(B)\ones_{f(B) = \varsigma}, \qquad \forall \varsigma \in \cB(M/(E-X)). \]
		
Then, we obtain \begin{align}\label{eq:sum=sum}
	\tau(M) &= \sum\limits_{B\in \cB} \lambda^*(B) &(\text{by definition of $\la^*$}) \notag\\ 
	&= \sum\limits_{\varsigma \in \cB(M/(E-X))}  \sum\limits_{B\in \cB} \lambda^*(B)\ones_{f(B) = \varsigma}& \notag\\
	&= \sum\limits_{\varsigma \in \cB(M/(E-X))} \omega(\varsigma).&
\end{align}
However, for every fixed $e \in X$:
\begin{align*}
\sum\limits_{\varsigma \in \cB(M/(E-X))} \omega(\varsigma) \ones_{e \in \varsigma} & = \sum\limits_{\varsigma \in \cB(M/(E-X))} \ones_{e \in \varsigma}\sum\limits_{B \in \cB(M)} \lambda^*(B)\ones_{f(B) = \varsigma} &\\
&  = \sum\limits_{\varsigma \in \cB(M/(E-X))}\sum\limits_{B \in \cB(M)} \lambda^*(B)\ones_{e \in \varsigma}\ones_{f(B) = \varsigma}&\\
&  = \sum\limits_{B \in \cB(M)} \sum\limits_{\varsigma \in \cB(M/(E-X))}\lambda^*(B)\ones_{e \in \varsigma}\ones_{f(B) = \varsigma}&\\
&  = \sum\limits_{B \in \cB(M)} \lambda^*(B)\sum\limits_{\varsigma \in \cB(M/(E-X))}\ones_{e \in \varsigma}\ones_{f(B) = \varsigma}&\\
& \leq \sum\limits_{B \in \cB(M)} \lambda^*(B)\ones_{e \in B} &(\text{since $f(B)\subset B$})\\
&\leq 1. &(\text{by definition of $\la^*$})
\end{align*}
Consequently, the vector $\omega$ satisfies all constraints of the base packing problem of $M/(E-X)$. By (\ref{eq:sum=sum}), we obtain that $\tau(M)	\leq \tau(M/(E-X))$.
		
\item Let $\kappa^*$ be an optimal solution for the base covering problem  (\ref{eq:base-covering-tree}) of $M$. For any $B \in \cB(M)$, by part 1 of Proposition \ref{pro:deletion}, we have that $B \cap Y$ is an independent set  in the matroid $M \setminus (E-Y)$. Then, for each $I \in \cI(M \setminus (E-Y))$, we define
\[ \nu(I) := \sum\limits_{B \in \cB(M)}  \kappa^*(B) \ones_{B \cap Y = I}.  \]
		
Then,\begin{align}\label{eq:sum=sum2}
\upsilon(M) &= \sum\limits_{B\in \cB(M)} \kappa^*(B) \notag\\
&= \sum\limits_{I \in \cI(M \setminus (E-Y))} \sum\limits_{B \in \cB(M)}  \kappa^*(B) \ones_{B \cap Y = I} \notag \\
& = \sum\limits_{I \in \cI(M \setminus (E-Y))} \nu(I).
\end{align}  
Moreover, we have that, for each $e\in Y$
\begin{align*}
\sum\limits_{ I \in \cI(M \setminus (E-Y))} \nu(I)\ones_{e \in I} & = \sum\limits_{I \in \cI(M \setminus (E-Y))} \ones_{e \in I}\sum\limits_{B \in \cB(M)}  \kappa^*(B) \ones_{B \cap Y = I}&\\
& = \sum\limits_{I \in \cI(M \setminus (E-Y))} \sum\limits_{B \in \cB(M)}  \kappa^*(B) \ones_{B \cap Y = I}\ones_{e \in I}&\\
& =\sum\limits_{B \in \cB(M)}\kappa^*(B)\sum\limits_{I \in \cI(M \setminus (E-Y))}   \ones_{B \cap Y = I}\ones_{e \in I}&\\
&=\sum\limits_{ B \in \cB} \kappa^*(B)\ones_{e\in B \cap Y}&\\
& = \sum\limits_{B \in \cB} \kappa^*(B)\ones_{e \in B}&\\
&\geq 1. &(\text{by definition of $\kappa^*$})
\end{align*}
Therefore, vector $\nu$ satisfies all constraints of the covering problem (\ref{eq:base-covering-tree}) of the matroid $M \setminus (E-Y)$. By (\ref{eq:sum=sum2}), we conclude that $\upsilon(M) \geq \upsilon(M \setminus (E-Y)).$
\end{enumerate}
\end{proof}
\begin{proof}[Proof for Theorem \ref{nash}]
Let $ E_{\min} $ and $ E_{\max} $ be defined as in (\ref{eq:min}) and (\ref{eq:max}). We have
\begin{align*}
	D(M) &= D(M \setminus (E - E_{\min})) &(\text{by part 4, part 5 of Theorem \ref{thm:min} and Corollary \ref{coro:homogeneous}}) \\
&	= \theta(M \setminus (E - E_{\min}))  &(\text{by part 5 of Theorem \ref{thm:min} and Corollary \ref{coro:homogeneous}} )\\
&= \upsilon(M \setminus (E - E_{\min})) &(\text{by Remark \ref{re:homo}}) \\
&\leq \upsilon(M) &(\text{by part 4 of Lemma \ref{lem:5sd}} )\\
&\leq D(M), &(\text{by (\ref{eq:longone})})
\end{align*}
and

\begin{align*}
	S(M) &= S(M / (E - E_{\max})) &(\text{by part 4, part 5 of Theorem \ref{thm:max} and Corollary \ref{coro:homogeneous}}) \\
	&=\theta(M / (E - E_{\max}))  &(\text{by part 5 of Theorem \ref{thm:max} and Corollary \ref{coro:homogeneous}} )\\
	&= \tau(M / (E - E_{\max})) &(\text{by Remark \ref{re:homo}}) \\
	&\geq \tau(M)  &(\text{by part 3 Lemma \ref{lem:5sd}} )\\
	&\geq S(M). &(\text{by (\ref{eq:longone})})
\end{align*}
Therefore, $\upsilon(M) = D(M) $ and $ \tau(M) = S(M) $.
\end{proof}
	
\section{Fulkerson duality for the base family}\label{sec:Fulkerson}
Given a matroid $ M(E,\cI) $, let $ \cB $ be the base family of $ M $. In this section, we introduce the Fulkerson blocker family and a Fulkerson dual family for the family $ \cB $.
	
For a matroid $ M(E,\cI) $, a set $ X \subseteq E $ is called a {\it separator} of $ M $ if any circuit $ C \in \cC(M) $ is contained in either $ X $ or $ E - X $. A matroid $ M(E,\cI) $ is called {\it connected} if it has no separators other than $ E $ and $ \emptyset $.
\begin{definition}\label{def:phi-theta}
Let $M$ be a matroid on $E$ with rank function $r$. Let $ \Phi $ be the family of all nonempty complement-closed sets $ X \subseteq E $ with usage vectors:
\begin{equation}\label{usage2}
\widetilde{\cN}(X,\cdot)^T = \frac{1}{r(E) - r(E - X)}\ones_{X}.
\end{equation}
Also, let $\Theta\subset\Phi$ be the family of all nonempty complement-closed sets $X \subseteq E$ with usage vectors as in (\ref{usage2})
for which the matroid $M / (E-X)$ is connected. 
\end{definition}
	
\begin{theorem}\label{theo:fulkerson}
Let $M = (E,\cI)$ be a loopless matroid with $r(M)>0$. Let $\cB$ be the base family of $M$ and let $\Theta$ be as in Definition \ref{def:phi-theta}. Then, the Fulkerson blocker family of $\cB$ is $\widehat{\cB} = \Theta$.
\end{theorem}
\begin{remark}By Proposition \ref{theo:smallestF}, the Fulkerson blocker family $\widehat{\cB}$ gives a  minimal inequality description $\Adm(\widehat{\cB})$ of the base dominant $\Dom(\cB)$. This is a common theme in the study of various convex polytopes.
\end{remark}
\begin{proof}[Proof of Theorem \ref{theo:fulkerson}]
The proof follows from three lemmas shown below.
By  Lemma \ref{lem:cl1}, $\widehat{\cB} \subset \Phi$. By Lemma \ref{lem:cl2}, $\Theta \subset \widehat{\cB}$. And by Lemma \ref{lem:cl3}, $ \Phi \cap \widehat{\cB} \subseteq \Theta $.
\end{proof}
\begin{lemma} \label{lem:cl1}
	We have  $\widehat{\cB} \subset \Phi$.
\end{lemma}
To prove Lemma \ref{lem:cl1}, we first recall the connection between the strength $S(M)$, the base packing value $\tau(M)$, and $\Mod_1(\cB)$. Note that, we have $S(M) = \tau(M) = \Mod_{1}(\cB)$ since the base packing problem is the dual problem of $\Mod_{1}(\cB)$.
For arbitrary weights $\si \in \mathbb{R}_{>0}^{E}$, we define the weighted strength  of $M$ with weights $\si$:
\begin{equation}\label{eq:weighted-strength-problem}
	S_{\si}(M) := \min \left\{ \frac{\si(X)}{r(E) - r(E-X)} : X \subseteq E, r(E) > r(E-X) \right\}.
\end{equation}
 We also define the weighted base packing problem:
 
 \begin{equation}\label{eq:weighted-packing-tree}
 	{\renewcommand{\arraystretch}{1.8}
 		\begin{array}{ll}
 			\underset{\lambda \in \R^{\cB}_{\geq 0}}{\rm maximize} & \sum\limits_{B \in \cB} \lambda(B) \\
 			\text{subject to} & \sum\limits_{B \in \cB : e \in B}\lambda(B) \leq \si(e), \qquad \forall e \in E.
 		\end{array}
 	}
 \end{equation}
 This is the dual problem of $\Mod_{1,\si}(\cB)$ and we denote its optimal value by $\tau_{\si}(M)$.
 
 \begin{lemma}\label{lem:strength-mod}
 	We have that $S_{\si}(M) =\tau_{\si}(M) = \Mod_{1,\si}(\cB)$ for any $\si \in \mathbb{R}_{>0}^{E}.$
 \end{lemma}
 \begin{proof}
 Let $\si \in \mathbb{Z}_{>0}^{E}.$ We assign each element $e$ with weight $\si(e)$. For each $e \in E(M)$, let $X_e = \lbr e_1,e_2,\dots, e_{\si(e)} \rbr$ be a set such that $X_e \cap X_{e'} = \emptyset,$ for all $ e, e' \in E(M)$ with $e \neq e'$. We can think of $X_e$ as the set of $\si(e)$ copies of $e$. The $\si$-parallel extension $M_{\si} $ of $M$ is obtained by replacing each element $e \in E(M)$ by $X_e$. Specifically, the ground set $E(M_{\si})$ of $M_{\si}$ is $\bigcup_{e \in E(M)}X_e$. A subset $Y \in E(M_{\si})$ is independent in $M_{\si}$ if and only if $\forall e \in E(M), |X_e \cap Y| \leq 1$ and the set $\lbr e \in E(M): X_e \cap Y \neq \emptyset \rbr$ is independent in $M$. If every element $e$ is assigned a constant weight $r$, we write $M_r$ for $M_{\si \equiv r}$ and call $M_r$ the $r$-parallel extension of $M$. Let $E' = \lbr e_1: e\in E(M)\rbr \subseteq E(M_{\si}).$ Consequently, there is a matroid isomorphism between $M$ and $M_{\si}|E'$ with the bijection $e \leftrightarrow e_1$ between $E(M)$ and $E'$. So, we shall see $M$ as the restriction $M_{\si}|E'$ of $M_{\si}$. Then, we have that $S_{\si}(M) = S(M_{\si})$ and $\tau_{\si}(M) = \tau(M_{\si})$.
Therefore, we establish that $S_{\si}(M) =\tau_{\si}(M) = \Mod_{1,\si}(\cB(M))$ for $\si \in \mathbb{Z}_{>0}^{E}.$ Note that $S_{\si}(M)$ and $\Mod_{1,\si}(\cB(M)) $ are homogeneous with respect to $\si$, meaning that $S_{r\si}(M) = rS_{\si}(M)$ and $\Mod_{1,r\si}(\cB(M)) = r\Mod_{1,\si}(\cB(M))$ for any positive number $r$. Therefore, $S_{\si}(M) =\tau_{\si}(M) = \Mod_{1,\si}(\cB(M))$ for $\si \in \mathbb{Q}_{>0}^{E}.$ And by continuity,  $S_{\si}(M) =\tau_{\si}(M) = \Mod_{1,\si}(\cB(M))$ for $\si \in \mathbb{R}_{>0}^{E}.$
\end{proof}
Using Lemma \ref{lem:strength-mod} and the Proof of Claim 4.1 in \cite{huyfulkerson}, we obtain Lemma \ref{lem:cl1}. Next, we consider the following well-known proposition for connected matroids.
\begin{proposition}\label{pro:connected}
A matroid $M$ is connected if and only if every two elements of $M$ lie in a common circuit of $M$.
\end{proposition}
Using Proposition \ref{pro:connected}, we derive the following lemma.
\begin{lemma}\label{lem:twobases}
Let $M(E,\cI)$ be a connected matroid. Then for any two elements $e_1$ and $e_2$ of $M$ there exist two bases $B_1$ and $B_2$ such that \[B_1 -B_2 = \lbr e_1 \rbr \quad \text{ and } \quad B_2 -B_1 = \lbr e_2 \rbr.\]
\end{lemma}
\begin{proof}
Let $C$ be a circuit that contains $e_1$ and $e_2$, then $C -\lbr e_2 \rbr$ is a independent set. There exists a base $B_1$ such that $(C -\lbr e_2 \rbr) \in B_1$. By part 3 of  Proposition \ref{pro:basic}, the set $B_2 := B_1 - \lbr e_1 \rbr \cup \lbr e_2 \rbr$ is a base of $M$.
\end{proof}
\begin{remark}\label{rem:base-admissibility}
Note that, by Lemma \ref{lem:BcapX}, for any $X\subset E$ such that  $r(E) - r(E - X)>0$, the vector in (\ref{usage2}) is necessarily admissible for $\cB$. Indeed,  if $B\in\cB$,
\[
\widetilde{\cN}(X,\cdot)\ones_B=\frac{1}{r(E) - r(E - X)}\ones_{X}^T \ones_B=\frac{|B\cap X|}{r(E) - r(E - X)}\ge 1.
\] 
\end{remark}
\begin{lemma}\label{lem:cl2}
We have  $\Theta \subset \widehat{\cB}$.
\end{lemma}
\begin{proof}
Every $X\in\Theta$ is non-empty and complement-closed, so $r(E) - r(E - X)>0$. Hence, by Remark \ref{rem:base-admissibility}, we have that $\Theta \subset \Adm(\cB)$.

Consider $X \in \Theta$, we aim to show that 
\[w :=\frac{1}{r(E) - r(E - X)}\ones_{X} \in \Ext(\Adm(\cB))=\widehat{\cB}.\]
Assume that there are two densities $\rho_1,\rho_2 \in \Adm(\cB)$ such that 
\[ w = \frac{1}{2}(\rho_1+\rho_2).\]
For every element $e \in E - X$, we have $\widetilde{\cN}(X,e)=0$, so 
\[\frac{1}{2}(\rho_1(e)+\rho_2(e))=w(e)= 0 \Rightarrow \rho_1(e)=\rho_2(e)=0.\] 
Now, assume that $|X| \geq 2$. Let $e_1$ and $e_2$ be two arbitrary distinct elements in $ X$. Since $M/(E-X)$ is connected, by Lemma \ref{lem:twobases}, there exists two bases $B_1,B_2$ of $M/(E-X)$ such that $B_1 -B_2 = \lbr e_1 \rbr$ and $B_2 -B_1 = \lbr e_2 \rbr $. By part 3 of Proposition \ref{pro:contraction}, there exists $B'_1 \in \cB(M)$ and $B'_2\in \cB(M)$ such that $B_1 =  B'_1 - (E-X)$ and $B_2 = B'_2 - (E-X)$. 
		
For any given density $\rho \in \R^E_{\geq 0 }$ and $A \subset E$, denote $\rho(A) := \sum\limits_{e\in A}\rho(e)$.
Since  $\rho_1,\rho_2 \in \Adm(\cB)$, we have
\[\rho_1(B'_1),\rho_1(B'_2) ,\rho_2(B'_1),\rho_2(B'_2) \geq 1.\] 
Since $\rho_1+\rho_2 = 2w$, we have  $\rho_1(B'_1)+\rho_2(B'_1) = 2w(B'_1) =  2w(B_1) = 2$ and  $\rho_1(B'_2)+\rho_2(B'_2) = 2w(B'_2) =2w(B_2) =  2$. This implies 
\[\rho_1(B'_1)=\rho_2(B'_1)=\rho_1(B'_2)=\rho_2(B'_2)=1.\]
Since  $\rho_1(e)=\rho_2(e)=0 $ for all $e \in E-X$, we obtain  
\[\rho_1(B_1)=\rho_2(B_1)=\rho_1(B_2)=\rho_2(B_2)=1.\]
Consequently, 
\begin{align*}
\rho_1(e_1)+\rho_1( B_1 \cap B_2)= \rho_1(e_2)+\rho_1( B_1 \cap B_2)= 1, \\
\rho_2(e_1)+\rho_2( B_1 \cap B_2)=\rho_2(e_2)+\rho_2( B_1 \cap B_2) =1.
\end{align*}
Hence, $\rho_1(e_1)=\rho_1(e_2)$ and $\rho_2(e_1)=\rho_2(e_2)$. Thus, since $e_1$ and $e_2$ were chosen arbitrarily, we conclude that $\rho_1$, $\rho_2$ are constant in $X$. Given that $\rho_1(B_1) =\rho_2(B_1) =1$ and $B_1$ has $r(E) -r(E-X)$ elements, we derive that 
		
\[\rho_1(e)= \rho_2(e) = \frac{1}{r(E) -r(E-X)}=w(e), \quad \forall e \in X.\]
Therefore, $\rho_1=\rho_2=w $.  So we have shown that $w$ is an extreme point of $\Adm(\cB)$.
		
If $|X|= 1$, suppose $X =\lbrace e \rbrace$. Let $B$ be a base $M$ that contains $e$. Using the same argument as above, we have $\rho_1(e) = \rho_2(e) = \frac{1}{r(E)-r(E-X)} = 1$. Therefore,  $\rho_1=\rho_2=w $, and once again, $w$ is an extreme point of $\Adm(\cB)$.	
In conclusion, $w \in \widehat{\cB}$.
\end{proof}
\begin{remark}
Note that the connectedness of $ M / (E - X) $ implies that $ X $ is complement-closed. This is because if $ E - X $ is not closed, then $ M / (E - X) $ will have some loops, which act as nontrivial separators.
\end{remark}
		
\begin{proposition}[\cite{mat}]\label{pro:se1}
For a matroid $ M(E,\cI) $, a set $ X \subseteq E $ is a separator of $ M $ if and only if $ r(X) + r(E - X) = r(E) $.
\end{proposition}
		
Using Proposition \ref{pro:se1}, we derive the following lemma:
\begin{lemma}\label{lem:cl3}
We have $ \Phi \cap \widehat{\cB} \subseteq \Theta $.
\end{lemma}
		
\begin{proof}
Let $X \in \Phi \cap \widehat{\cB}$. Assume that the matroid $M/(E-X)$ is not connected. Then, let $A$ be a nontrivial separator of $M/(E-X)$ and define $D := X-A$.
Using part 4 of Proposition \ref{pro:contraction}, we derive that
\begin{equation}\label{eq:rank-contract-A}
r_{M / (E-X)}(A) = r_{M}(E-D) - r_{M}(E-X),
\end{equation}
\begin{equation*}
r_{M / (E-X)}(D) = r_{M}(E-A) - r_{M}(E-X),
\end{equation*}
and
\begin{equation*}
r_{M / (E-X)}(X) = r_{M}(E) - r_{M}(E-X).
\end{equation*}
		
Since $X$ is nonempty and complement-closed, $E-X$ is closed in $M$. Therefore, since $E-D \supsetneq E-X$, we deduce $r(E-D) > r(E-X)$, which, by (\ref{eq:rank-contract-A}), implies that $r_{M / (E-X)}(A) >0$. Similarly, we also have  $r_{M / (E-X)}(D) >0$. Consequently, by Proposition \ref{pro:se1}, we have
\begin{equation}\label{eq:positive-rank-D}
 r_M(E) -  r_M(E-D) = r_{M / (E-X)}(X) - r_{M / (E-X)}(A) = r_{M / (E-X)}(D) >0,
 \end{equation}
\begin{equation}\label{eq:positive-rank-A}
r_M(E) -  r_M(E-A) = r_{M / (E-X)}(X) - r_{M / (E-X)}(D) = r_{M / (E-X)}(A) >0,
\end{equation}
and adding these two identities and using Proposition \ref{pro:se1} again, we obtain
\begin{equation}\label{eq:sum-of-ranks}
 2r_M(E) -  r_M(E-A)  -  r_M(E-D) = r_{M / (E-X)}(X)  = r_M(E) - r_M(E-X).
 \end{equation}
Therefore, we have that $\frac{1}{r_M(E) - r_M(E - X)}\ones_{X}$ is equal to
\begin{align*}
\frac{r_M(E) - r_M(E -A)}{r_M(E) - r_M(E - X)}\cdot \frac{1}{r_M(E) - r_M(E - A)}\ones_{A} + \frac{r_M(E) - r_M(E - D)}{r_M(E) - r_M(E - X)}\cdot\frac{1}{r_M(E) - r_M(E - D)}\ones_{D}.
\end{align*}
Hence, by (\ref{eq:sum-of-ranks}), we see that $\frac{1}{r_M(E) - r_M(E - X)}\ones_{X}$ is a convex combination of $\frac{1}{r(E) - r(E - A)}\ones_{A}$ and $\frac{1}{r(E) - r(E - D)}\ones_{D}$ and these two vectors are in $\Adm(\cB)$, by (\ref{eq:positive-rank-D}) and (\ref{eq:positive-rank-A}) and Remark \ref{rem:base-admissibility}. This is a contradiction, since $X \in \widehat{\cB}$.
\end{proof}

\section{Modulus for dual matroids}\label{sec:dualmatroid}
	
Given a loopless matroid $ M(E,\cI) $ with $ r(E) > 0 $, let $ \cB $ be the base family of $ M $. We recall that the set
\[\cB^* = \left\{ X \subseteq E : \text{there exists a base } B \in \cB \text{ such that } X = E - B \right\}\]
is the family of bases of the dual matroid $ M^* $ of $M$ with the rank function $ r^* $. We also have that $ r^*(E) = |E| - r(M) $. The dual of the dual of a matroid $ M $ is the matroid $ M $ itself, in other words, $ (M^*)^* = M $.
	
For a loopless matroid $ M $, the dual matroid $ M^* $ is not necessarily loopless. If we consider the base modulus for a matroid with loops, the optimal element usage probabilities $ \eta^* $ of those loops are zero because they are not contained in any bases. Furthermore, the dual matroid $ M^* $ does not necessarily have positive rank. Specifically, $ r^*(E) > 0 $ if and only if $ r(E) < |E| $.

Note that for a loopless matroid $ M(E,\cI) $, by Lemma \ref{lem:positive}, we have $ 0 < \eta^*(e) \leq 1 $ for all $ e \in E $, where $ \eta^* $ are the optimal element usage probabilities. In this section, we investigate the relationships between the base modulus of a matroid and the base modulus of its dual. Consider a loopless matroid $ M(E,\cI) $ with $ r(E) > 0 $. If $ r(E) = |E| $, then the dual matroid $ M^* $ has rank zero. Now, assuming $ r(E) < |E| $, the concept of dual matroids motivates the following theorem.
	
\begin{theorem}\label{thm:dualmat}
Let $ E $ be a finite set. Let $ \Gamma $ be a family of non-empty proper subsets of $ E $ of equal size with usage vectors given by indicator functions. Define $ \Gamma^* := \{ E - \gamma : \gamma \in \Gamma \} $ with usage vectors given by indicator functions.  Let $\widehat{\Ga}$ and $\widehat{\Ga^*}$ be the Fulkerson blocker families of $\Ga$ and $\Ga^*$, respectively. Let $ \eta^* $ and $ \eta_{\circ}^* $ be the optimal densities for $ \Mod_2(\widehat{\Gamma}) $ and $ \Mod_2(\widehat{\Gamma^*}) $, respectively. Then we have 
\[\eta^* + \eta_{\circ}^* = \mathbf{1},\]
where $ \mathbf{1} $ is the vector of all ones.
\end{theorem}
	
\begin{proof}
Let $\cN$ be the usage matrix of $\Gamma$. Then, all row sums of $\cN$ are equal to $k$. Let $\one_{|\Gamma| \times |E|}$ be the $|\Gamma| \times |E|$ matrix of all ones. Consequently, $\one_{|\Ga| \times |E|} - \cN$ is the usage matrix of $\Ga^*$.
		
Note that by the proof of Lemma \ref{lem:fixedsum}, we have 
\begin{equation}\label{eq:sum-eta}
	(\eta^*)^T\one = k,
\end{equation} 
and
\begin{equation}\label{eq:sum-etastar}
 (\eta_{\circ}^*)^T\one = |E| - k.
\end{equation}
Additionally, we have 
\begin{equation}\label{eq:N-ones}
	\cN \one = k\one, \quad 
	(\one_{|\Ga| \times |E|})\eta^* = k\one, \quad 
	(\one_{|\Ga| \times |E|})\eta_{\circ}^* = (|E| - k)\one.
\end{equation}

By Theorem \ref{thm:meomod}, we have that $\rho^* = \eta^*/\cE(\eta^*)$ is the optimal solution for $\Mod_2(\Ga)$ and  $\rho_{\circ}^* = \eta_{\circ}^*/\cE(\eta_{\circ}^*)$ is the optimal solution for $\Mod_2(\Ga^*)$. Hence, by admissibility, we have 
\begin{equation}\label{eq:adm1}
\cN\frac{\eta^*}{\cE(\eta^*)} \geq \one,
\end{equation}
and
\begin{equation}\label{eq:adm2}
\left( \one_{|\Ga| \times |E|} - \cN\right)\frac{\eta_{\circ}^*}{\cE(\eta_{\circ}^*)} \geq \one,
\end{equation}
where these two inequalities hold element-wise.
Let $a = \one- \eta_{\circ}^* \leq \one$, then
\begin{align*}
\left( \one_{|\Ga| \times |E|} - \cN\right)\frac{\eta_{\circ}^*}{\cE(\eta_{\circ}^*)} &=\left( \one_{|\Ga| \times |E|} - \cN\right)\frac{\one-a}{\cE(\one-a)} &\\  
&=\frac{|E|\one - k\one -k\one+\cN a}{\cE(\one-a)} &(\text{by (\ref{eq:N-ones})})\\
& =\frac{(|E|-2k)\one+\cN a}{|E| -2k + \cE(a)},&(\text{by (\ref{eq:sum-etastar})})
\end{align*}
where in the denominator we expanded the squares and used (\ref{eq:sum-etastar}).
Therefore, (\ref{eq:adm2}) is equivalent to 
\begin{equation*}
(|E|-2k)\one+\cN a \geq (|E| -2k + \cE(a))\one
\end{equation*}
which we rewrite as, 
\begin{equation*}
\cN\frac{ (\one -\eta_{\circ}^*)}{\cE(\one -\eta_{\circ}^*)} =	\frac{\cN a}{\cE(a)} \geq \one,
\end{equation*}  
Similarly, the inequality (\ref{eq:adm1}) is equivalent to 
\begin{equation*}
\left( \ones_{|\Ga| \times |E|} - \cN\right)\frac{\one - \eta^*}{\cE(\one - \eta^*)} \geq \one.
\end{equation*}
In other words, the vector $\frac{ \one -\eta_{\circ}^*}{\cE(\one -\eta_{\circ}^*)}$ is admissible for $\Ga$ and $\frac{\one - \eta^*}{\cE(\one - \eta^*)}$ is admissible for $\Ga^*$. Hence, since $\rho^*$ is optimal for $\Mod_2(\Ga)$, we have
\begin{align}\label{eq:e1}
\frac{1}{\cE(\one -\eta_{\circ}^*)} &= \cE \left(\frac{ \one -\eta_{\circ}^*}{\cE(\one -\eta_{\circ}^*)} \right) \geq \cE(\rho^*) = \frac{1}{\cE(\eta^*)},
\end{align}
and, since $\rho_{\circ}^*$ is optimal for $\Mod_2(\Ga^*)$,
\begin{align}\label{eq:e2}
\frac{1}{\cE(\one -\eta^*)} &= \cE \left(\frac{ \one -\eta^*}{\cE(\one -\eta^*)} \right) \geq \cE(\rho_{\circ}^*) = \frac{1}{\cE(\eta_{\circ}^*)}.
\end{align}
These are equivalent to
\begin{align*}
	\cE(\eta^*) \geq \cE(\one - \eta_{\circ}^*) &= |E| -2(|E|-k) +\cE(\eta_{\circ}^*),& (\text{by (\ref{eq:sum-etastar})})
\end{align*}
and 
\begin{align*}
	\cE(\eta_{\circ}^*) \geq \cE(\one - \eta^*) &= |E| -2k +\cE(\eta^*).& (\text{by (\ref{eq:sum-eta})})
\end{align*}
Consequently, we have shown that		
\begin{equation*}
\cE(\eta^*)  - k \geq \cE(\eta_{\circ}^*) -(|E|-k) \geq	\cE(\eta^*)  - k.
\end{equation*}
		
Therefore, two inequalities in (\ref{eq:e1}) and (\ref{eq:e2}) holds as equalities. Hence, by uniqueness of $\eta^*$ and $\eta_{\circ}^*$, we obtain that $\one - \eta_{\circ}^* = \eta^*$.	
\end{proof}

Applying this to the case when $\Ga$ is the base family of a matroid, we recover a known result \cite{catlin1992} about strength and fractional arboricity of dual matroids.
\begin{corollary}
Let $M=(E,\cI)$ be a loopless matroid and $0 <r(M)< |E|$. Let $M^*$ be the dual matroid of $M$. Let $\eta^*$ be the optimal density for $\Mod_2(\widehat{\cB(M)})$ and let $\eta_{\circ}^*$ be the optimal density for $\Mod_2(\widehat{\cB(M^*)})$. Then, we have \[\eta^* + \eta_{\circ}^* =\one,\] where $\one$ is the vector of all ones.
Moreover, if the dual matroid $M^*$ is loopless, we have 
\begin{equation*}
\frac{1}{D(M)}+\frac{1}{S(M^*)} = 1,
\end{equation*}
and
\begin{equation*}
\frac{1}{S(M)}+\frac{1}{D(M^*)} = 1.
\end{equation*}
\end{corollary}
	
\begin{proof}
Given that $M$ is loopless, we have
\[\eta^*_{\min} = \frac{1}{D(M)} \quad \text{and} \quad \eta^*_{\max} = \frac{1}{S(M)}.\]
If $M^*$ is also loopless, we get
\[(\eta_{\circ}^*)_{\min} = \frac{1}{D(M^*)} \quad \text{and} \quad (\eta_{\circ}^*)_{\max} = \frac{1}{S(M^*)}.\]
The proof follows by applying Theorem \ref{thm:dualmat}.
\end{proof} 
	
\section{Deflation process and other values of $p$}\label{sec:otherp}
Let $M(E,\cI)$ be a matroid with base family $\cB =\cB(M)$.  We want to show that $\Mod_p(\cB)$ can be deduced from $\Mod_2(\cB)$. 
\begin{theorem}
	Let $M(E,\cI)$ be a matroid with base family $\cB =\cB(M)$.  Let $\widetilde{\cB}$ be a Fulkerson dual family of $\cB$. Let $p \in (1, \infty) \setminus \lbr 2\rbr.$ Let $q$ be the Hölder conjugate exponent of $p$, so that $(p-1)(q-1) = 1$. Then  $\Mod_2(\widetilde{\cB})$ and  $\Mod_q(\widetilde{\cB})$ have the same optimal density $\eta^*$. Moreover,
	
	\begin{equation*}
		\Mod_p(\cB) =  \left(\sum\limits_{e \in E} (\eta^*(e))^q \right)^{1-p}.
	\end{equation*}
	
\end{theorem}

\begin{proof}
Let $\eta^*$ be the unique optimal density for $\Mod_2(\widetilde{\cB})$.
We recall $E_{min}$ defined as in (\ref{eq:min}). According to part 1 of Theorem \ref{thm:min}, $E_{min}$ is a complement-Beurling set. Our objective is to iteratively apply Theorem \ref{thm:serialmodmod} to the Beurling set $E-E_{min}$. Initially, we decompose the matroid $M$ into $M_1 := M|E_{min}$ and $M/E_{min}$ by utilizing Theorem \ref{thm:serialmodmod}. Then, $M_1$ is homogeneous, we proceed to further decompose the matroid  $M/E_{min}$ using the minimum value of $\eta^*$, continuing this deflation process. Then, this deflation process
yields the principal partitions of the matroid $M$. Let us denote $M_1,M_2,\dots, M_k$ as the sequence of homogeneous matroids resulting from this process, alongside the corresponding sequence of ground sets $E_1, E_2,\dots, E_k$.

Note that $E_1, E_2,\dots, E_k$ are subsets of $E$, and $\eta^*$ is equal to the constant $\eta^*_j := \frac{r(M_j)}{|E_j|}$ on each $E_j$, where $r(M_j)$ is the rank of matroid $M_j$. Moreover, we have $\eta^*_1 <\eta^*_2 <\dots <\eta^*_k$. Then, we establish that

\begin{equation}
	(\Mod_2(\cB(M)))^{-1} = \sum\limits_{j=1}^{k} \frac{r(M_j)^2}{|E_j|}.
\end{equation}
 The $q$-energy of $\eta^*$ is
\begin{equation}\label{eq:energyq}
	\cE_q(\eta^*) = \sum\limits_{j=1}^{k} |E_j| \left( \frac{r(M_j)}{|E_j|}\right)^q.
\end{equation} 
 We define a density $\rho$ by the formula
 \begin{equation*}
 	\rho := \frac{(\eta^*)^{q-1}}{\cE_q(\eta^*)}.
 \end{equation*}
 We want to check that $\rho \in \Adm(\cB)$. Note that for every fair base $B \in \cB$, we have $|B \cap E_j| = r(M_j)$ for every $j$. This is because since $E_{1}$ is a complement-Beurling set, we have $|B \cap E_1| = r_M(E_1) = r(M_1)$ (by Lemma \ref{lem:BcapX} and part 3 of Proposition \ref{pro:deletion}). Then $B \cap (E-E_{1})$ is a base in $M/E_{1}$ (by part 3 of Proposition \ref{pro:contraction}), and it is also a fair base in $M/E_{1}$ (by Theorem \ref{thm:serialmodmod} (ii)). By induction, we achieve that $|B \cap E_j| = r(M_j)$ for every $j$. Therefore, for any fair base $B$, we have 
\begin{align}\label{eq:total-fair}
	\ell_{\rho}(B) &= \frac{1}{\cE_q(\eta^*)} \sum\limits_{j=1}^{k} r(M_j) \left( \frac{r(M_j)}{|E_j|}\right)^{q-1} \notag\\
	& =\frac{1}{\cE_q(\eta^*)} \sum\limits_{j=1}^{k} |E_j| \left( \frac{r(M_j)}{|E_j|}\right)^q \notag\\
	&=1,
\end{align}
where the last equality follows from (\ref{eq:energyq}).

Let $B$ be a base in $\cB(M)$ that has minimum total usage $\ell_{\rho}(B)$. We have that $|B \cap E_1| \leq r_{M}(E_1) = r(M_1)$ by Lemma \ref{lem:BcapX}. Assume that $|B \cap E_1| < r_{M}(E_1)$. 
Let $T$ be a maximal independent set in $E_1$, then $|T| = r_M(E_1)  > |B \cap E_1|$. Note that  $B \cap E_1$  and $T$ are independent and $(B \cap E_1) \subset E_1$. By the exchange property (I3) of Definition \ref{def:independent-set}, there exists $z \in (T - (B \cap E_1)) = T - B$ such that $(B \cap E_1) \cup \{z\}$ is independent in $M$.
Since $B$ is a base and $z \notin B$, let $C$ be the unique circuit within $B \cup \{z\}$, containing $z$. Since $(B \cap E_1) \cup \{z\}$ is independent in $M$, it follows that $C \not\subseteq (B \cap E_1) \cup \{z\}$. Consequently, $C - E_1 \neq \emptyset$. Let $x$ be an element in $C - E_1$. Then
$\rho(x) > \rho(z)$ by definitions of $\rho$ and $E_1$. Note that $x \in C$ and $x \neq z$, then $x \in B$.
By part 3 of Proposition \ref{pro:basic}, we have $B' := (B - \{ x \}) \cup \{ z \}$ is a base and, recalling the total usage from (\ref{eq:total-usage}), we have
\[\ell_{\rho}(B') = \ell_{\rho}(B) - \rho(x) + \rho(z) < \ell_{\rho}(B).\] This contradicts the definition of $B$. Therefore, we have that $|B \cap E_1| = r(M_1)$. Then $B \cap (E-E_{1})$ is a base in $M/E_{1}$ (by part 3 of Proposition \ref{pro:contraction}). Then, we use the same argument, by induction, we achieve that $|B \cap E_j| = r(M_j)$ for every $j$. Therefore, by (\ref{eq:total-fair}), we have $\ell_p(B)=1$. Thus, $\rho \in \Adm(\cB)$.

Next, since $\rho \in \Adm(\cB)$ and $\eta^* \in \Adm(\widetilde{\cB})$, we have
\begin{align}\label{eq:modq-energyp}
	\Mod_p(\cB) &\leq \cE_p(\rho) \\
	&=\sum\limits_{e\in E} \rho(e)^p =\frac{1}{\cE_q(\eta^*)^p}\sum\limits_{e\in E} (\eta^*(e))^{qp-p}\notag\\
	&=\frac{1}{\cE_q(\eta^*)^p}\sum\limits_{e\in E} (\eta^*(e))^{q}=\cE_q(\eta^*)^{1-p}\notag\\
	&\leq \left( \Mod_q(\widetilde{\cB})\right)^{1-p}.\notag
\end{align}
But Fulkerson duality for $p$-modulus says that 
\begin{equation*}
	\left( \Mod_q(\widetilde{\cB})\right)^{1-p} = \left( \Mod_p(\cB)\right)^{-(q-pq)/p} = \left( \Mod_p(\cB)\right).
\end{equation*} 
Therefore, equality holds in (\ref{eq:modq-energyp}) and we have shown that $\rho$ is extremal for $\Mod_p(\cB)$. By (\ref{eq:weighted-eta-rho}), we have $\eta^*$ is the optimal density for $\Mod_q(\widetilde{\cB})$. In particular, we have the following formula

\begin{equation*}
	\Mod_p(\cB) = \left(\sum\limits_{j=1}^k \frac{r(M_j)^q}{|E_j|^{q-1}} \right)^{1-p}.
\end{equation*}
\end{proof}

\bibliographystyle{acm}
\bibliography{Modulusforbasesofmatroids}

\begin{thebibliography}{10}

\bibitem{modulus}
{\sc Albin, N., Brunner, M., Perez, R., Poggi-Corradini, P., and Wiens, N.}
\newblock Modulus on graphs as a generalization of standard graph theoretic
  quantities.
\newblock {\em Conform. Geom. Dyn. 19\/} (2015), 298--317.

\bibitem{pietroblocking}
{\sc Albin, N., Clemens, J., Fernando, N., and Poggi-Corradini, P.}
\newblock Blocking duality for {$p$}-modulus on networks and applications.
\newblock {\em Ann. Mat. Pura Appl. (4) 198}, 3 (2019), 973--999.

\bibitem{pietrofairest}
{\sc Albin, N., Clemens, J., Hoare, D., Poggi-Corradini, P., Sit, B., and
  Tymochko, S.}
\newblock Fairest edge usage and minimum expected overlap for random spanning
  trees.
\newblock {\em Discrete Math. 344}, 5 (2021), Paper No. 112282, 24.

\bibitem{polynomial}
{\sc Albin, N., Kottegoda, K., and Poggi-Corradini, P.}
\newblock A polynomial-time algorithm for spanning tree modulus.
\newblock {\em arXiv: Combinatorics\/} (2020).

\bibitem{pietrosecure}
{\sc Albin, N., Kottegoda, K., and Poggi-Corradini, P.}
\newblock Spanning tree modulus for secure broadcast games.
\newblock {\em Networks 76}, 3 (2020), 350--365.

\bibitem{pietrominimal}
{\sc Albin, N., and Poggi-Corradini, P.}
\newblock Minimal subfamilies and the probabilistic interpretation for modulus
  on graphs.
\newblock {\em J. Anal. 24}, 2 (2016), 183--208.

\bibitem{bertsimas}
{\sc Bertsimas, D., and Tsitsiklis, J.~N.}
\newblock {\em Introduction to linear optimization}, vol.~6.
\newblock Athena scientific Belmont, MA, 1997.

\bibitem{modern}
{\sc Bollob\'{a}s, B.}
\newblock {\em Modern graph theory}, vol.~184 of {\em Graduate Texts in
  Mathematics}.
\newblock Springer-Verlag, New York, 1998.

\bibitem{catlin1992}
{\sc Catlin, P.~A., Grossman, J.~W., Hobbs, A.~M., and Lai, H.-J.}
\newblock Fractional arboricity, strength, and principal partitions in graphs
  and matroids.
\newblock {\em Discrete Applied Mathematics 40}, 3 (1992), 285--302.

\bibitem{cunninghamfaster}
{\sc Cheng, E., and Cunningham, W.~H.}
\newblock A faster algorithm for computing the strength of a network.
\newblock {\em Information Processing Letters 49}, 4 (1994), 209--212.

\bibitem{chopraon}
{\sc Chopra, S.}
\newblock On the spanning tree polyhedron.
\newblock {\em Oper. Res. Lett. 8}, 1 (1989), 25--29.

\bibitem{cunninghamoptimal}
{\sc Cunningham, W.~H.}
\newblock Optimal attack and reinforcement of a network.
\newblock {\em J. Assoc. Comput. Mach. 32}, 3 (1985), 549--561.

\bibitem{edmonds1965lehman}
{\sc Edmonds, J.}
\newblock Lehman's switching game and a theorem of tutte and nash-williams.
\newblock {\em J. Res. Nat. Bur. Standards Sect. B 69\/} (1965), 73--77.

\bibitem{edmondsgeedy}
{\sc Edmonds, J.}
\newblock Matroids and the greedy algorithm.
\newblock {\em Math. Programming 1\/} (1971), 127--136.

\bibitem{fujishige1980lexicographically}
{\sc Fujishige, S.}
\newblock Lexicographically optimal base of a polymatroid with respect to a
  weight vector.
\newblock {\em Mathematics of Operations Research 5}, 2 (1980), 186--196.

\bibitem{fujishige2009theory}
{\sc Fujishige, S.}
\newblock Theory of principal partitions revisited.
\newblock {\em Research Trends in Combinatorial Optimization: Bonn 2008\/}
  (2009), 127--162.

\bibitem{fulkersonblocking}
{\sc Fulkerson, D.~R.}
\newblock Blocking polyhedra.
\newblock In {\em Graph {T}heory and its {A}pplications ({P}roc. {A}dvanced
  {S}em., {M}ath. {R}esearch {C}enter, {U}niv. of {W}isconsin, {M}adison,
  {W}is., 1969)\/} (1970), Academic Press, New York, pp.~93--112.

\bibitem{fulkersonanti}
{\sc Fulkerson, D.~R.}
\newblock Blocking and anti-blocking pairs of polyhedra.
\newblock {\em Math. Programming 1\/} (1971), 168--194.

\bibitem{gusfieldtree}
{\sc Gusfield, D.}
\newblock Connectivity and edge-disjoint spanning trees.
\newblock {\em Inform. Process. Lett. 16}, 2 (1983), 87--89.

\bibitem{hong2016fractional}
{\sc Hong, Y., Gu, X., Lai, H.-J., and Liu, Q.}
\newblock Fractional spanning tree packing, forest covering and eigenvalues.
\newblock {\em Discrete Applied Mathematics 213\/} (2016), 219--223.

\bibitem{closure82}
{\sc Matthews, L.}
\newblock Closure in independence systems.
\newblock {\em Math. Oper. Res. 7}, 2 (1982), 159--171.

\bibitem{nash1964decomposition}
{\sc Nash-Williams, C. S.~J.}
\newblock Decomposition of finite graphs into forests.
\newblock {\em Journal of the London Mathematical Society 1}, 1 (1964), 12--12.

\bibitem{edge-disjoint}
{\sc Nash-Williams, C. S. J.~A.}
\newblock Edge-disjoint spanning trees of finite graphs.
\newblock {\em J. London Math. Soc. 36\/} (1961), 445--450.

\bibitem{mat}
{\sc Pitsoulis, L.~S.}
\newblock {\em Topics in matroid theory}.
\newblock SpringerBriefs in Optimization. Springer, New York, 2014.

\bibitem{rucinski1986strongly}
{\sc Ruci{\'n}ski, A., and Vince, A.}
\newblock Strongly balanced graphs and random graphs.
\newblock {\em Journal of graph theory 10}, 2 (1986), 251--264.

\bibitem{huyfulkerson}
{\sc Truong, H., and Poggi-Corradini, P.}
\newblock Fulkerson duality for modulus of spanning trees and partitions.
\newblock {\em arXiv preprint arXiv:2306.15984\/} (2023).

\bibitem{on}
{\sc Tutte, W.~T.}
\newblock On the problem of decomposing a graph into {$n$} connected factors.
\newblock {\em J. London Math. Soc. 36\/} (1961), 221--230.

\bibitem{fastapproximation}
{\sc Worou, B. M.~T., and Galtier, J.}
\newblock Fast approximation for computing the fractional arboricity and
  extraction of communities of a graph.
\newblock {\em Discrete Applied Mathematics 213\/} (2016), 179--195.

\end{thebibliography}
\def\cprime{$'$}
\nocite{*}
	
\end{document}